\documentclass[a4paper,12pt]{article}
\usepackage{amsmath,amsthm,amssymb,enumerate,array}

\usepackage[all]{xy}
\usepackage[pdftex]{hyperref}
\usepackage{qtree}

\theoremstyle{plain}

\newtheorem{Thm}{Theorem}[section]
\newtheorem{Prop}[Thm]{Proposition}

\newtheorem{Lem}[Thm]{Lemma}
\newtheorem{Cor}[Thm]{Corollary}
\theoremstyle{remark}
\newtheorem{Rem}[Thm]{Remark}

\theoremstyle{definition}
\newtheorem{Ex}[Thm]{Example}

\newtheorem{Not}[Thm]{Notation}
\newtheorem{Def}[Thm]{Definition}
\newtheorem{Defs}[Thm]{Definitions}

\newcommand{\tripod}[3]{\raisebox{4ex}{ \Tree [.{#2}
      [. {#1} {#3} ] ]} }
\newcommand{\nopod}[3]{\ensuremath{\raisebox{3ex}{ \Tree [.{#2}  {#1}  {#3}  ]} }}

\newcommand{\quot}[2]{\ensuremath{\raisebox{-.4ex}[0pt][0pt]#1 \hspace{0.1ex}\backslash \hspace{-0.1ex}\raisebox{0.2ex}[0pt][0pt]#2}}

\newcommand{\simfk}{\ensuremath{\langle\sim_{f^{k}}\rangle}}

\newcommand{\R}{\mathbb{R}}
\newcommand{\Z}{\mathbb{Z}}

\def\coloneqq{\mathrel{\mathop\mathchar"303A}\mkern-1.2mu=}

\DeclareMathOperator{\dist}{d}
\DeclareMathOperator{\Aut}{Aut}
\DeclareMathOperator{\Out}{Out}
\DeclareMathOperator{\Inn}{Inn}
\DeclareMathOperator{\Hyp}{Hyp}
\DeclareMathOperator{\Lip}{Lip}
\DeclareMathOperator{\Coll}{Coll}

\DeclareMathOperator{\axis}{axis}
\title{Stretching factors, metrics and train tracks for free products}
\author{Stefano Francaviglia, Armando Martino}

\begin{document}

\maketitle

\begin{abstract}
In this paper we develop the metric theory for the outer space of a free product
of groups. This generalizes the theory of the outer space of a free group, and includes
its relative versions. The outer space of a free product is made of $G$-trees with possibly
non-trivial vertex stabilisers. The strategies are the same as in the classical case, with some
technicalities arising from the presence of infinite-valence vertices.

We describe the Lipschitz metric and show how to compute it; we prove the
existence of optimal maps; we describe geodesics represented by folding paths.

We show that train tracks representative of irreducible (hence hyperbolic) automorphisms exist
and that their are metrically characterized as minimal displaced points, showing in particular
that the set of train tracks is closed (in particular answering to some questions raised
in~\cite{HM} concerning the axis bundle of irreducible automorphisms).

Finally, we include a proof of the existence of simplicial train tracks map without
using Perron-Frobenius theory.

A direct corollary of this general viewpoint is an easy proof that relative train track maps exist in both the free group and free product case.
\end{abstract}
\tableofcontents

\section{Introduction}
In this paper we are interested in studying the outer space of a free product of groups.
Namely, given a group of the form $G=G_1*\dots*G_p*F_k$, we study the set of trees where $G$
acts with vertex stabilizers the $G_i$'s. In the case $G=G_1*\dots*G_p*F_k$ is the
free product decomposition of a finitely generated group $G$, then this was introduced by Guirardel and
Levitt in~\cite{GuirardelLevitt} in the case where this is the Grushko decomposition of $G$. That is, when each $G_i$ is freely
indecomposable and not isomorphic to $\Z$. However, we shall consider such spaces with respect to an arbitrary free product decomposition,
and not necessarily the natural Grushko one. Similar spaces are studied by Sykiotis in~\cite{Syk}.

The theory is similar to that of the case of free
groups, with the advantage that this unified viewpoint covers at once both the general case of a
free product as well as many ``relative'' cases of the classical Outer space. The Group of
isomorphisms that acts on $G$ will be that of automorphisms that preserve the set of conjugacy
classes of the $G_i$, (which coincides with $\Aut(G)$ in the case of the Grushko decomposition,
by the Kurosh subgroup theorem).

In particular, one can define the Lipschitz metric
(see~\cite{BestvinaYael,FrancavigliaMartino,FrancMart} for the classical case). The presence of
vertex stabilizers involves some technical complications (for instance, the Ascoli-Arzel\`a theorem
does not hold for spaces that are not locally compact) but the main results of the classical
case hold {\em mutatis mutandis}. For instance, optimal maps exist and Lipschitz factors can
be computed on a list of simple candidates. Also, geodesics are constructed via folding paths.

For the study of automorphisms a very useful tool is the theory of train track maps, developed
by Bestvina and Handel~\cite{BestvinaHandel} (see also~\cite{BFH1,BFH2,BFH3,BFH4})  and
extensively used in literature. This tool is available also in the present setting.

For studying train tracks, we chose to follow the metric viewpoint as in~\cite{BestvinaBers}.
In particular, we show that for an irreducible automorphism
the set of train tracks coincides with the set of minimally displaced elements. We
remark that there is no uniform definition of train track maps in the literature, even if the difference
from one definition to another is minimal. As the set of minimally displaced elements is
closed, this gives in particular a proof that the set of train tracks is closed, hence
answering to a question raised in~\cite{HM}, where the authors give a characterizations of the axis
bundle of an irreducible automorphism (see Remark~\ref{r22}).
We would also like to mention the very recent preprint~\cite{MP} about axis bundles.

Many of the results about train tracks that we are going to describe are well known (at least to
the experts) in the case
of free groups, and the proofs in our general setting do not require substantial changes. We
give here explicit and fully detailed proofs of all these facts for completeness. We refer the reader
also to the very recent and nice preprint~\cite{seb} that deals with local finite trees with
possibly non trivial edge-stabilizers, from
the same viewpoint of us. As S. Meinert pointed out, the fact the we work with trivial
edge-stabilizers is crucial, as Lemma~\ref{lem:Omapsexists} may fail in general. In this work
we do not develop the theories of geodesic currents and laminations  (\cite{CH,CHL1,CHL2,CHL3,CHL4,CHL5,CHR,Fra,Ha,K1,K2,KL1,KL2,KL3,KN1,KN2,KN3}) that would certainly be
of interest in this general case.

\textsc{Acknowledgments}. We are grateful to Yago Antolin Pichel and Camille Horbez
for the many interesting
discussions we had with them. Many thanks go to Sebastian Meinert for its helpful comments on a
previous version of this manuscript.
We want to thank the CRM of Barcelona and the LABEX of Marseille
for having hosted great research periods in geometric group theory, where we had the occasion
to discuss the subject of the present paper with so many great people.
We are clearly in debt with the organizers of such events.

We would also like to thank the Instituto Nazionale di Alta Matematica for support during the period in which this research was carried out.

\section{$G$-trees and lengths}\label{s:gtree}
For any simplicial tree $T$ (not necessarily locally compact), we denote by $VT$ and $ET$ the set of vertices and edges of $T$
respectively. A {\it simplicial metric tree}, is a simplicial tree equipped with a complete path
metric such that edges are isometric to closed intervals of $\mathbb R$. Note that the
simplicial structure on a metric tree is an additional structure that is not necessarily determined by the metric structure. However, we do require that all branch points be vertices, and generally we will simply take the set of vertices to be the set of branch points (which is determined by the metric structure).

For $x,y\in T$, we denote by $[x, y]_T$ (or simply by $[x,y]$ if there is no ambiguity concerning $T$) the unique path from $x$ to $y$, and for a path $\gamma$ in $T$
we denote by $l_T(\gamma)$ the length of $\gamma$ in $T$.

Let $G$ be a group. In this work, by a {\it $G$-tree} we mean a simplicial metric tree $T=(T,d_T)$,
where $G$ acts simplicially on $T$ and for all $g\in G$ and  $e\in ET$, $e$ and $ge$ are isometric.
In other words, $G$ is acting on $T$ by isometries and preserving the simplicial structure.

If $T$ is a $G$-tree then the quotient space $\quot{G}{T}$ is a graph. We denote by
$\pi_T:T\to \quot{G}{T}$ the projection map.

In general, a path $\gamma$ in $\quot{G}{T}$ may have
many lifts to $T$, even if we fix the initial point of the lift. This is because
each time $\gamma$ passes through an edge whose initial vertex has a lift with non-trivial stabilizer, we have many choices for the lift of the edge.

Let $T$ be a $G$-tree.  The following definitions depend on the action of $G$ on $T$.
An element $g\in G$ is called {\it hyperbolic} if it fixes no points. Any hyperbolic
element $g$ of $G$ acts by translation on a subtree of $T$ homeomorphic to the real line, called the
{\it axis } of $g$ and denoted by {\em $axis_T(g)$}. The {\it translation length} of $g$ is the distance that $g$ translates the
axis. The action of $G$ on $T$ defines a {\it length function} denoted by $l_T$ $$ l_T\colon
G\to \mathbb{R},\quad l_T(g)\colon= \inf_{x\in T}\dist_T(x,gx).$$
\begin{Rem}\label{rem:lengthg}
We note that, under our hypothesis, this $\inf$ is always achieved (see for example
\cite[1.3]{CullerMorgan}). In particular, $g\in G$ is hyperbolic if and only if $l_T(g)>0$.
\end{Rem}


\section{The Outer space of a free product}
We follow \cite{GuirardelLevitt}. We will consider groups $G$ of the form
$$G=G_1*\dots*G_p*F_k$$
where $(G_i)_{i=1}^p$ is a family of groups, and
$F_k$ denotes the free group of rank $0 \leq k < \infty$.

We will be mainly concerned with the case where $G$ admits a co-compact action on a tree with trivial edge stabilisers
and indecomposable vertex stabilisers, or equivalently a group of finite Kurosh rank. However, in general we will not assume that
that $G_i$ are indecomposable. That is, while $G$ may admit a decomposition as a free product of finitely many freely indecomposable groups,
we are interested in developing the subsequent theory in the situation where our given free
product decomposition is not necessarily of that kind. For instance, we will apply the theory in the case that $G$ is free, and the $G_i$ are
certain free factors of $G$.

Let $\mathcal{T}(G)$ denote the set of simplicial metric $G$-trees.
We say that two elements $T,T'$ of $\mathcal{T}(G)$ are {\em equivalent}, and we  write $T\sim T'$,
 if there
exists a $G$-equivariant isometry $f:T\to T'$.

Let $T\in\mathcal T(G)$. A vertex $v\in VT$ is {\it redundant}, if it has degree two, and any $g$
that fixes $v$ also fixes the edges adjacent to $v.$ It is {\it terminal} if $T-\{v\}$ is
connected. We will consider $G$-trees with no {\it redundant} vertices.

Let $\mathcal{O}=\mathcal{O}(G,(G_i)_{i=1}^p,F_k)$ be the subset of $\mathcal{T}(G)/\sim$ of simplicial, metric $G$-trees $T$, up to equivariant isometry, satisfying that
\begin{enumerate}
\item[(C0)] $T$ has no redundant vertices;
\item[(C1)] the $G$-action of $T$ is {\it minimal} (i.e there exist no proper invariant subtree), with trivial edge stabilizers;
\item[(C2)] for each $i=1,\dots, p,$ there is exactly one orbit of vertices with stabilizer conjugate to $G_i$ and all edge stabilizers are trivial;
\item[(C3)] all other vertices have trivial stabilizer. We will often refer to such vertices as {\em
    free} vertices.
\end{enumerate}

It may be worth to mention that under such assumptions, for any $T\in\mathcal O$ the quotient
$\quot GT$ is a finite graph.

The space $\mathcal{O}$ admit a natural action of $(0,\infty)$ defined by rescaling the metric, that
is to say, multiplying all lengths of the edges by the same number. The quotient space of
$\mathcal{O}$ by that action
is denoted by $\mathcal{PO}=\mathcal{PO}(G,(G_i)_{i=1}^p,F_k)$ and is called the {\it outer space } of
$G$. Sometimes $\mathcal{O}$
will be referred to as the {\it unprojectivized outer space} of $G$.

There is a natural map from $\mathcal{T}(G)$ to $\mathbb{R}^G$, mapping $T$ to
$(l_T(g))_{g\in G}$.
This map clearly factors through $\mathcal{T}(G)/\sim$. The following fact is proved
in~\cite[Thm 3.7]{CullerMorgan}
\begin{Lem}\label{lempaulin}
The restriction of the translation length function to $\mathcal O \to \mathbb R^G$ is injective.
\end{Lem}

The {\it axes topology} on $\mathcal{O}$ is the topology induced as a subspace of
$\mathbb{R}^G.$

As in \cite{GuirardelLevitt}, there are in fact two topologies on $\mathcal{O}$.
There is, as in Outer Space, the simplicial topology which is different from the Gromov topology
(which coincides with the axes topology). The metric we study in the following discussion induces the same topology as the
axes topology.

\begin{Def}
  The group $\Aut(G,\mathcal O)$ is the group of automorphisms that preserve the set of
  conjugacy classes of the $G_i$'s. Namely $\phi\in \Aut(G)$ belongs to $\Aut(G,\mathcal O)$ if
  $\phi(G_i)$ is conjugate to one of the $G_i$'s.
\end{Def}
In the case of the Grushko decomposition $\Aut(G)=\Aut(G,\mathcal O)$.
The group $\Aut(G,\mathcal O)$ acts on $\mathcal{T}(G)$ by changing the action.
That is, for $\phi\in \Aut(G)$ and $T$ in $\mathcal{T}(G),$ the image of $T$ under $\phi$  is the
$G$-tree with the same underlying tree as $T$, endowed with the action given by $(g,x)\in G\times
T\mapsto \phi(g) x\in T$.
If $\phi_h$ is the automorphism of $G$ given by conjugation by $h\in G,$ ($g\mapsto h^{-1}gh$), then for
every $T\in \mathcal{T}(G),$ $T\sim \phi_h(T)$ via the map $T\to \phi_h(T),$ $x\mapsto h^{-1}x.$
Thus $\Out(G,\mathcal O)=\Aut(G,\mathcal O)/\Inn(G)$ acts on $\mathcal{T}(G)/\sim.$


\section{The Metric}
\subsection{$\mathcal{O}$-Maps}
Let $T$ be a $G$-tree. Denote by $\Hyp(T)$ set of elements $g\in G$ whose the action on $T$ is
hyperbolic (see \cite{CullerMorgan} for details).  If $T\in \mathcal{O}$ and $g\notin\Hyp(T),$ then $g$ fixes a vertex of $T,$ and by (C2)
there exits
$i\in \{1,\dots, p\}$ such that $g$ lies in a $G$-conjugate of $G_i.$ Conversely, if $g$ lies in a
$G$-conjugate of some $G_i$, $i\in \{1,\dots, p\},$ by (C2) $g$ fixes a vertex, and then it is not
hyperbolic. Therefore, $g\in G$ is hyperbolic for $T\in\mathcal O$ if and only if it is hyperbolic
for any other element of $\mathcal O$. The set of hyperbolic elements of $G$ for some (and
hence for all) $T$ in $\mathcal{O}$ is denoted by $\Hyp(\mathcal{O})$.

\begin{Def}[$\mathcal{O}$-maps]
Let $A,B\in \mathcal{O}.$ An {\it $\mathcal{O}$-map}  $f\colon A\to B$ is a $G$-equivariant,
Lipschitz continuous, surjective function. (Note that we don't require to $f$ to be a graph
morphism). We denote by $\Lip(f)$ the {\it  Lipschitz constant  of $f,$} that is the smallest
constant $K\geq 0$ such that, for all $x_1,x_2\in A$ $$\dist_B(f(x_1),f(x_2))\leq K
\dist_A(x_1,x_2).$$
\end{Def}

\begin{Lem}\label{lem:Omapsexists}
For every pair $A,B\in \mathcal{O},$ there exists a $\mathcal{O}$-map  $f\colon A\to B$. Moreover,
any two $\mathcal O$-maps from $A$ to $B$ coincide on the non-free vertices.
\end{Lem}
\begin{proof}
Let $A, B$ be two $G$-trees. Let $v$ be a non-free vertex $A$ with stabilizer
 $\mathrm{stab}(v)=H< G$. By $(C2)$  $H$ is conjugate to one of the $G_i$'s, and again by $(C2)$
there exist a unique vertex $w$ of $B$ which is fixed by $H$. Define $f(v)=w$. Do the same for all
the non free-vertices of $A$. The map defined so far on
non-free vertices is
equivariant because $v$ is stabilized by $H$ if and only if $gv$ is stabilized by $gHg^{-1}$.
It follows that $f(gv)=gw=gf(v)$. Note that this argument also proves the second claim.

Now, extend the map $f$  equivariantly on the orbits of free vertices without requiring any other
condition. Note that each orbit of a free vertex is simply isomorphic to $G$, as a $G$-set, and we simply map each free $G$ orbit of vertices to another free $G$ orbit of vertices. However, note that we do \emph{not} require that distinct orbits map to distinct orbits.

We have now defined an equivariant map on all the vertices of $A$. Each component of the complement of the vertices is an (open) edge, and the $G$ action is free on the set of edges. Therefore, we may define the map linearly on the edges and this will clearly be equivariant. Thus we have defined an equivariant
map which is Lipschitz continuous because $G$-trees of $\mathcal O$ have only finitely many orbits of
vertices and edges. Moreover, the map $f$ is surjective because its image is a $G$-invariant
sub-tree of $B$, that must be $B$ due to $(C1)$. Thus $f$ is an $\mathcal O$-map.
\end{proof}

\begin{Lem}\label{lem:minn}
Let $A,B\in\mathcal{O}.$
For any $\mathcal O$-map $f\colon A\to B$, we have
$$\sup_{g\in \Hyp(\mathcal{O})}\dfrac{l_B(g)}{l_A(g)}\leq \Lip(f) $$
\end{Lem}
\begin{proof}

Let $f\colon A\to B$ be an $\mathcal{O}$-map, $g\in G,$ and  $x\in A.$ Since $f$ is continuous
$[f(x), f(gx)]_B\subseteq f([x, gx]_A),$  then
\begin{align*}
l_B(g) \leq &\dist_B(f(x),gf(x))\\
=&\dist_B(f(x), f(gx))\\
\leq& l_B(f([x,gx]_A))\\
\leq& \Lip(f) l_A([x,gx]_A).
\end{align*}
By Remark \ref{rem:lengthg}, there exists  $x_g\in A$ realizing $l_A(g),$ that is
$l_A([x_g,gx_g]_A)=l_A(g).$ Using $x_g$ in the previous inequality, we conclude that
\begin{equation}
l_B(g) \leq \Lip(f)l_A(g).
\end{equation}
\end{proof}

\subsection{The Metrics}
\begin{Defs}
For any pair $A,B\in \mathcal{O}$ we define the right and left maximal stretching factors
$$\Lambda_R(A,B)\coloneqq \sup_{g\in \Hyp(\mathcal{O})}\dfrac{l_B(g)}{l_A(g)}\quad
\Lambda_L(A,B)\coloneqq \sup_{g\in \Hyp(\mathcal{O})}\dfrac{l_A(g)}{l_B(g)}=\Lambda_R(B,A)$$
and {\em asymmetric pseudo-distances}
$$\dist_R(A,B)\coloneqq \log\Lambda_R(A,B)\qquad
\dist_L(A,B)\coloneqq \log\Lambda_L(A,B)=\dist_R(B,A).$$

We define $\Lambda(A,B)\coloneqq \Lambda_R(A,B)\Lambda_L(A,B)$ and the {\em distance} between
$A$ and $B$ as
$$\dist(A,B)\coloneqq \log \Lambda(A,B).$$
\end{Defs}

Directed triangular inequalities are readily checked for $d_R$ and $d_L$, thus triangular inequality
holds for $d$. Moreover, $d$ is a genuine distance on $\mathcal O$ as $d(A,B)=0$ gives
$l_B(g)=l_A(g)$ for any element of $G$, and this implies that $A=B$ by Lemma~\ref{lempaulin}.
The functions $d_R$ and $d_L$ become asymmetric distances once restricted to the subset of
$\mathcal O$ of $G$-trees with co-volume one, which can be identified with $\mathcal{PO}$.
(See for example \cite{FrancavigliaMartino,FrancMart,BestvinaYael} for the study of such
functions in the case of outer space of free groups.)

\begin{Lem}
The action of $\Aut(G,\mathcal O)$ on $\mathcal O$ is by isometries.
\end{Lem}
\proof If $\phi\in\Aut(G,\mathcal O)$, the it preserves the conjugacy classes of the
$G_i$'s. Therefore $g$ is hyperbolic if and only if $\phi(g)$ is. Thus
$$\sup_{g\in \Hyp(\mathcal{O})}\dfrac{l_B(g)}{l_A(g)}=
\sup_{g\in \Hyp(\mathcal{O})}\dfrac{l_B(\phi(g))}{l_A(\phi(g))}=
\sup_{g\in \Hyp(\mathcal{O})}\dfrac{l_\phi(B)(g)}{l_\phi(A)(g)}.
$$
\qed


\section{Equivariant Ascoli-Arzel\'a}
In this section we provide a tool for computing stretching factors. We follow the approach
of~\cite{FrancavigliaMartino}. The main issue is that given $A,B\in\mathcal O$,
one needs to find a map between them
which optimize the Lipschitz constant. Since elements of $\mathcal O$ are not locally compacts,
Ascoli-Arzel\'a does not apply directly, and we need to control local pathologies by hand.

The lazy reader may skip this section by paying the small price of missing out on some definitions and the
beautiful proof of the equivariant version of Ascoli-Arzel\'a theorem.

\begin{Def}
A map $f:A\to B$ between metric graphs is called {\em piecewise linear}  if it is continuous
and
 for all edges $e$ of $A$, there exists a positive number $S_{f,e}$, called the
stretching factor of $f$ at $e$, such that the restriction of $f$ to $e$
 has constant speed $S_{f,e}$. More precisely, $f$ is piecewise linear if for any $e\in EA$, the following diagram
commutes and the vertical functions are local isometries:
\begin{figure}[h]
\centerline{
\xymatrix{
e\subset A \ar[rr]^{f_{|_e}}\ar[d]& & f(e)\subset B\\
\mathbb R\supset [0,l_A(e)]\ar[rr]_{t\mapsto S_{f,e}t}& &[0,S_{f,e}l_A(e)]\subset\mathbb R\ar[u]
}}
\end{figure}
\end{Def}
We remark that piecewise linear maps are locally injective on edges.
\begin{Def}[PL-map for trees]
Let $A,B\in \mathcal O$. We say that a function $f:A\to B$ is a {\it PL-map} if it is a piecewise
linear $\mathcal{O}$-map.
For any $\mathcal{O}$-map $f:A\to B$ we define the map $PL(f)$ as the unique $PL$-map that coincides
with $f$ on
vertices.
\end{Def}
\begin{Rem}
Let $f:A\to B$ be an $\mathcal O$-map and $e\in EA$.
 If $l_B(f(e))$ denotes the distance between the images of the vertices of $e$, then
by construction we have $S_{PL(f),e}=l_B(f(e))/l_A(e)\leq
\Lip(f_{|_e})\leq\Lip(f)$, for all $e\in EA.$ Therefore, $\Lip(PL(f))\leq \Lip(f).$
\end{Rem}

Before proving the equivariant Ascoli-Arzel\'a, we discuss an example.

\begin{Ex}

Consider a segment $[0,3]$ with free vertices and a segment $[0,1]$ with a vertex non free, say $0$, with associated group $\mathbb Z$. Consider the associated
trees $A=[0,3]$ and $B$. $B$ is a star-shaped tree with an infinite valence vertex, say $0$,
from which emanate infinite copies of $[0,1]_n$ labeled by $n\in\mathbb Z$.
Now consider the map $f:[0,3]\to [0,1]$
$$f(t)=
\left\{\begin{array}{ll}
  t& t\in[0,1]\\ 2-t&t\in[1,2]\\ t-2&t\in[2,3]
\end{array}
\right.$$

For any $n\in\mathbb Z$ there exist a lift $f_n:A\to B$ of $f$ such that the segment $[0,2]$ is
mapped to $[0,1]_n$ and $[2,3]$ is mapped to $[0,1]_0$. The sequence $f_n$ has no sub-sequence
that converges, but clearly if one ``straightens'' $f_n$ by collapsing $[0,2]$ to $0$, this
sequence becomes constant. Of course, this is safe because there is no $G$-action on $A$.
\end{Ex}

This is more or less everything that can go wrong. We now introduce the precise notion of collapsible and non-collapsible map.
\begin{Def}\label{def:collapsible}
Let $A\in\mathcal O$.
A subset $X\subset A$  is {\em collapsible} if $gX\cap X=\emptyset$ for any $Id\neq g\in G$.
\end{Def}

\begin{Def}
  Let $A,B\in\mathcal O$ and $f:A\to B$ be an $\mathcal O$-map. A {\em collapsible component}
  of $f$ is a connected component of $A\setminus f^{-1}(v)$, for a $v\in VB$ non-free, which is
  collapsible.
\end{Def}

\begin{Def}
  Let $A,B\in \mathcal O$ and $f:A\to B$ be an $\mathcal{O}$-map.
   $f$ is said {\em collapsible} if it has a collapsible component.
  $f$ is said {\em non-collapsible} if it is not collapsible.
\end{Def}

Note that $f$ is non-collapsible if any component $C$ of $A\setminus f^{-1}(v)$ either
contains a non-free vertex or there is point $w\in A$ and $id\neq g\in G$ so that both $w$ and
$gw$ belong to $C$.

\begin{Def}
  Let $A,B\in\mathcal O$ and $f:A\to B$ be an $\mathcal O$-map. $f$ is {\em $\sigma$-PL} if
  \begin{itemize}
  \item $\sigma$ is a simplicial structure $(V\sigma,E\sigma)$ on $A$ obtained by adding $2$-valent vertices to $A$.
  \item The number of $G$-orbits of edges of $\sigma$ is finite.
  \item For any $v\in VB$ non-free, $f^{-1}(v)$ is a forest (union of trees) with leaves in $V\sigma$.
  \item $f$ is PL w.r.t. $\sigma$.
  \end{itemize}
\end{Def}

Any $PL$-map $f:A\to B$ is $\sigma$-PL for the pull-back structure induced on $A$ by $\sigma$.

\begin{Lem}\label{lem:colfinite}
  Let $A,B\in\mathcal O$ and $f:A\to B$ be a $\sigma$-PL map. Then the number of  orbits of
  collapsible components of $f$ is finite.
\end{Lem}
\proof
First note that $A\setminus f^{-1}(gv)=g(A\setminus f^{-1}(v))$. Hence, the orbits of components
corresponding to the orbit of $v$ have representatives in $A\setminus f^{-1}(v)$.
 Since there are finitely
many orbits of vertices, it is enough to show that
the collapsible components in $A\setminus f^{-1}(v)$ are contained in finitely many orbits.

We argue by contradiction and assume that we have infinitely many collapsible components $C_i$
of $A\setminus f^{-1}(v)$ in distinct orbits.

Since there are finitely many orbits of edges, we may assume that the orbit of some edge $e$
meets every $C_i$; hence there are $g_i\in G$ such that $g_ie\in C_i$. Moreover, for the same
reason and from the definition of collapsible component, we deduce that there is a
uniform bound on the  number of edges in any collapsible component.

Without loss of generality we may assume that the number of edges of $C_0$ is maximal amongst the
$C_i$ and that $g_0=Id$. Since $C_i$ and $C_0$ are not in the same orbit, $g_iC_i\neq C_0$. On
the other hand $g_iC_i\cap C_0\neq\emptyset$, thus one of them contains a leaf of the other.
Since $C_0$ is maximal there is a leaf $x_i$ of $g_iC_i$ in $C_0$.
Leaves of $C_i$ are $\sigma$-vertices, and since $C_0$ has finitely many vertices and edges, we may
assume that $x_i=x$ is independent of $i$, and that there is an edge $\xi$ of $C_0$ contained
in $g_iC_i$ for all $i$ --- note that $g_iC_i\cap C_0$ contains at least one edge because it is
the intersection of open sets ---

As $x$ is a leaf of $g_iC_i$, $f(x)=g_iv$ for all $i$. In particular $$g_i^{-1}g_j(v)=v$$
Since $\xi\subset g_iC_i\cap g_jC_j$ we have $C_i\cap g_i^{-1}g_j C_j\neq\emptyset$. However
$C_i$ is a component of $A\setminus f^{-1}(v)$ and $g_i^{-1}g_jC_j$ is a component of
$A\setminus f^{-1}(g_i^{-1}g_jv)=A\setminus f^{-1}(v)$. Hence they are equal contradicting the
fact that the $C_i$'s are in distinct orbits.\qed

\begin{Lem}\label{lem:decorbit}
  Let $A,B\in\mathcal O$. Let $f:A\to B$ be a collapsible $\sigma$-PL map. Then there is
  an $\mathcal O$-map $f_\bullet:A\to B$ such that:
  \begin{itemize}
  \item $f_\bullet$ is $\sigma$-PL (same $\sigma$).
  \item $\Lip(f_\bullet)\leq \Lip(f)$.
  \item The number of orbits of collapsible components of $f_\bullet$ is strictly smaller than that of $f$.
  \end{itemize}
\end{Lem}
\proof Let $v\in VB$ non-free and let $C$ be a collapsible component of $A\setminus
f^{-1}(v)$. Collapse $C$ by defining $f_\bullet|_C=v$. Extend $f_\bullet$ by equivariance on the orbit of
$C$. This is possible since $gC\cap C=\emptyset$ for $g\neq id$. On the remaining part of $A$
let $f_\bullet=f$. Clearly $f_\bullet$ is an $\mathcal O$-map which is $\sigma$-PL and satisfies
$\Lip(f_\bullet)\leq\Lip(f)$.

Since $g(A\setminus f^{-1}(v))=A\setminus f^{-1}(gv)$, it follows that $A\setminus f^{-1}(v)$
contains a representative for every orbit of components.
In passing from $f$ to $f_\bullet$, the components of $A\setminus f^{-1}(v)$ which are not of the form
$gC$ are unchanged, while the orbit of
$C$ is removed. More precisely, $A\setminus f_\bullet^{-1}(v)=\{A\setminus f^{-1}(v)\}\setminus
GC$.
Thus the number of orbits of collapsible components in $A\setminus f^{-1}(v)$
is decreased by $1$.

Now, consider the non-free vertices of $B$ that are not in the orbit of $v$ and chose
orbit-representatives $w_1,\dots,w_k$. Define $G$-sets
$$U_i=\{D\ :\ D\text{ is a component of }A\setminus f^{-1}(gw_i)\text{ for some }g\in G\}$$
$$\widehat U_i=\{D\ :\ D\text{ is a component of }A\setminus f_\bullet^{-1}(gw_i)\text{ for some }g\in G\}$$
$$V_i=\{D\ :\ D\text{ is a collapsible component of }A\setminus f^{-1}(gw_i)\text{ for some
}g\in G\}$$
$$\widehat V_i=\{D\ :\ D\text{ is a collapsible component of }A\setminus f_\bullet^{-1}(gw_i)\text{ for some
}g\in G\}$$

Since $w_i\notin Gv$, then $f_\bullet^{-1}(w_i)\subseteq
f^{-1}(w_i)$.  Therefore, any component $K$ of $A\setminus f^{-1}(w_i)$ is contained in a
unique component
$K_\bullet$ of $A\setminus f_\bullet^{-1}(w_i)$. Moreover, if $K_\bullet$ is
collapsible, so is $K$.

This inclusion defines a $G$-equivariant surjection $\iota: U_i\to \widehat U_i$ such that $\widehat
V_i\subset \iota (V_i)$. Therefore the number of $G$-orbits in $V_i$ is greater than or equal
to the number of $G$-orbits in $\widehat V_i$.
\qed

\begin{Cor}[Existence of Coll]\label{cor:def}
  Let $A,B\in\mathcal O$. Let $f:A\to B$ be a $\sigma$-PL map. Then there is
  an $\mathcal O$-map $\Coll(f):A\to B$ such that:
  \begin{itemize}
  \item $\Coll(f)$ is $\sigma$-PL (same $\sigma$).
  \item $\Lip(\Coll(f))\leq \Lip(f)$.
  \item $\Coll(f)$ is non-collapsible.
  \end{itemize}
 \end{Cor}
\proof This follows by induction from Lemmas~\ref{lem:colfinite} and~\ref{lem:decorbit}.\qed

In the sequel we use the following conventions:
\begin{itemize}
\item When we write $\Coll(f)$ we mean any map given by Corollary~\ref{cor:def}.
\item We say that $P$ is true eventually on $n$ if $\exists n_0$ so that $P$ is true for all
  $n>n_0$, and we write $P$ is true $\forall n>>0$.
\item $P$ is true frequently if $\forall n\exists m>n$ so that $P$ is true for $m$.
\item A sequence {\em sub-converges} if it converges up to passing to sub-sequences.

\end{itemize}

Now we are in position to prove the existence of a map that minimizes the Lipschitz factor.

\begin{Thm}[Equivariant Ascoli-Arzel\'a]\label{thm:finfty}
Let $A,B\in \mathcal{O},$ then there exits a PL-map $F\colon A\to B$ with
$$\Lip(F)=\inf\{\Lip(\varphi): \varphi\text{ is an  $\mathcal{O}$-map from } A \text{ to } B \}.$$
\end{Thm}
\proof For the entire proof --- which requires several lemmas--- we fix  a minimizing sequence
$f'_n:A\to B$ of PL-maps so that $$\lim_{n\to\infty}Lip(f'_n)=\inf\{\Lip(\varphi): \varphi\text{ is an  $\mathcal{O}$-map from } A \text{ to } B \}$$
and we define
$$f_n=\Coll(f'_n).$$

By definition of $\Coll$ we have that the $f_n$ are non-collapsible, uniformly $L$-Lipschitz and
$$\lim_{n\to\infty}Lip(f_n)=\inf\{\Lip(\varphi): \varphi\text{ is an  $\mathcal{O}$-map from } A \text{ to
} B \}.$$

By Ascoli-Arzel\'a the maps $\pi_B\circ f_n\circ\pi_A^{-1}:\quot{G}{A}\to\quot{G}{B}$ sub-converge to a map $\bar
f_\infty$. We will show that $\bar f_\infty$ is in fact the projection of a map $A\to B$ which
is the limit of $f_n$. From now on we restrict to a sub-sequence and we suppose that
$\pi_B\circ f_n\circ\pi_A^{-1}$ uniformly converges to $\bar f_\infty$.

Let $\mathcal T$ be the set of pairs $(T,f)$ such that
\begin{itemize}
\item $T\subset A$ is a $G$-invariant subset of $A$ (not necessarily simplicial).
\item $f:T\to B$ is $G$-equivariant and $L$-Lipschitz.
\item $\pi_B(f(t))=\bar f_\infty(\pi_A(t))$ for any $t\in T$.
\item $f_n|_T$ sub-converges to $f$.
\end{itemize}

The set $\mathcal T$ is ordered by inclusion/consistency:
$(T,f)<(Q,u)$ if $T\subset Q$ and $f=u|_T$. (Note that $\mathcal T\neq\emptyset$, because
$f_n$ is constant on non-free vertices.)

We need a couple of standard facts on Lipschitz functions, that we collect in the following
lemma whose proof is left to the reader.
\begin{Lem}\label{lem:standard} Let $X\subset Y$ be metric spaces and let $Z$ be a complete metric space. Denote by
  $\bar X$ the closure of $X$ in $Y$. Then
\begin{enumerate}
\item If $f:X\to Z$ is a $L$-Lipschitz map, then there is a $L$-Lipschitz map $\bar f:\bar X\to
  Z$ so that $\bar f|_X=f$.
\item If $u_n: X\to Z$ is a sequence of $L$-Lipschitz maps and $u_\infty:\bar X \to Z$ is such that $u_n\to u_\infty$ on $X$, then the
extensions $\bar u_n:\bar X \to Z$ converge to $u_\infty$.
\item Suppose in addiction that $\bar X$ is compact, then the point-wise convergence of $u_n$
  is uniform.
\end{enumerate}
\end{Lem}

\begin{Lem}\label{lem:uniform}
If $(T,f)\in \mathcal T$ then $f_n|_{T}$ sub-converges uniformly to $f$.
\end{Lem}
\proof We restrict to the sub-sequence where $f_n|_T$ sub-converges. $G$ acts by
isometries on $A,B$. $\bar T$ is $G$-invariant and admits a compact fundamental domain $K$.
Since $f_n$ are uniformly Lipschitz, we can apply Lemma~\ref{lem:standard}, point $3$ to $K$
and get uniform convergence on $K$. The uniform convergence on $T$ follows from
$G$-equivariance of $f_n$ and $f$.
\qed

If $\{(T_i,\varphi_i)\}$ is a chain in $\mathcal T$
then, by Lemma~\ref{lem:uniform} and a standard argument on sub-sequences, $(\cup_i T,\cup_i\varphi _i)$ is an upper
bound. Therefore $\mathcal T$ has a maximal element.

Let $(T,f_\infty)$ be a maximal element of $\mathcal T$. If we show that $T=A$ we are done
because $f_\infty=\lim f_n$ realizes the minimum Lipschitz constant and $F=PL(f_\infty)$ will
be PL and with the same Lipschitz constant.

\begin{Lem}
 $T$ contains all non-free vertices and it is closed.
\end{Lem}
\proof Both claims follow from maximality of $T$. The first is because $G$-equivariance
implies that $f_n(v)=Fix_B(Stab_A(v))$ is a constant sequence.
The second is an immediate consequence of Lemma~\ref{lem:standard}.
\qed

Assuming that $T\neq A$ and maximal we shall derive a contradiction.
Let $x\in\partial T \cap A$ be fixed for the remainder of the proof. As $T$ is closed $x\in T$. Define
$$\lambda_A=\min_{w\in VA \ :\ w\neq x} d_A(x,w) \qquad
\lambda_B=\min_{w\in VB \ :\ w\neq f_\infty(x)} d_B(f_\infty(x),w)$$
$$ \lambda=\min(\lambda_A,\lambda_B/2L)$$
Choose $y\notin T$ such that $d_A(x,y)<\lambda$.

Since $f_n(x)\to f_\infty(x)$ eventually on $n$ we have
$$f_n(y)\in B(f_\infty(x),\lambda_B)$$
Note that  $B(f_\infty(x),\lambda_B)$ is star-shaped, namely it contains at most one vertex and
contains exactly one vertex if and only if $f_\infty(x)$ is a vertex of $B$.

If $f_n(y)$ sub-converges, we can extend $f_\infty$ to $y$ and then extend equivariantly
contradicting the maximality of $T$. Therefore $f_n(y)$ does not sub-converge. In particular,
this implies that $B(f_\infty(x),\lambda_B)$ does not have compact closure, hence
$$f_\infty(x)=v$$ is non-free vertex of $B$. Also, for the same reason $f_n(y)\neq v$ eventually on
$n$, and so after passing to a sub-sequence we may assume that $f_n(y)\neq v\ \forall n$.

Define
$$C_n(y)=\text{ the connected component of } A\setminus f^{-1}_n(v)\text{ containing }y$$

The rest of the argument is devoted to proving that $C_n(y)$ is collapsible eventually on $n$. This
contradicts the fact that $f_n$ are not collapsible and completes the proof.

\begin{Lem}\label{lem:closurecompact}
Let $C\subset A\setminus f_n^{-1}(v)$ be a connected subset such that there is a $w\in C$ and
$Id\neq g\in G$ with $gw\in C$.
Then, there exist two connected component $D_1(g)$ and $D_2(g)$ of $B\setminus v$,
depending only on $g$, such that $$f_n(C)\subset D_1(g)\cup D_2(g).$$
\end{Lem}
\proof Since $C$ is connected $f_n(C)$ is contained in a single component of $B\setminus v$,
the point is the independence from $n$.

Set $Inv_B(g)$ to be $\axis_B(g)$ if $g$ is hyperbolic and $Inv_B(g)=Fix_B(g)$ if $g$ is
elliptic. $Inv_B(g)$ is either a line or a single point. Therefore, it intersects at most two components of $B\setminus v$ that we denote $D_1(g)$
and $D_2(g)$ (possibly $D_1(g)=D_2(g)$).

The segment $[w,gw]_A$ is contained in $C$ and $f_n([w,gw]_A)\cap Inv_B(g)\neq\emptyset$.
Therefore $f_n(C)$ is contained in $D_1(g)\cup D_2(g)$.\qed

\begin{Lem}\label{lem:almostnoncoll}
  For any $Id\neq g\in G$, we have $gC_n(y)\cap C_n(y)=\emptyset$ eventually on $n$.
\end{Lem}
\proof We argue by contradiction. So there exist $g$ such that $gC_n(y)\cap
C_n(y)\neq\emptyset$ frequently on $n$. By Lemma~\ref{lem:closurecompact} $f_n(C)\subset
D_1(g)\cup D_2(g)$ frequently on $n$. Therefore $f_n(y)\in
B(v,\lambda_B)\cap(D_1(g)\cup D_2(g))$ which is the union of at most two open segments and hence
has compact closure. In particular $f_n(y)$ sub-converges contradicting the maximality of
$T$.\qed

Lemma~\ref{lem:almostnoncoll} is not enough to conclude that $C_n(y)$ is collapsible because
a priori for any given $n$ there may be infinitely many $g$  such that $gC_n(y)\cap C_n(y)\neq\emptyset$.

For any half-line starting form $y$ let $w$ the first vertex such that there is $Id\neq g_w\in G$
with
$g_ww\in[y,w]$. Let $K$ be the union of all such segments. By construction $K$ is a simplicial
tree containing $y$. Also, the diameter of $K$ is finite because there are finitely many orbit
of vertices. Moreover, the interior of $K$ does not contain any non-free vertex. Therefore, $K$
is a finite simplicial tree.

\begin{Lem}\label{lemma:intK}
  Eventually on $n$ we have $C_n(y)\subset int(K)$.
\end{Lem}
\proof Since $K$ is finite the collection $G_K=\{g_w\ :\ w\text{ a leaf of } K\}$ is finite. By
Lemma~\ref{lem:almostnoncoll} for every $g_w\in G_K$, $g_wC_n(y)\cap C_n(y)=\emptyset$ eventually
on $n$. Up to passing to a sub-sequence we may suppose that this happens for any $n$. Therefore
$C_n(y)$ cannot contain the segment $[g_ww,w]$. Since $C_n(y)$ is connected and contains $y$,
it follows that it does not contain any of the leaves of
$K$. The claim follows.\qed

\begin{Lem}\label{lem:intK2}
  There are only finitely many $g\in G$ such that $g(int(K))\cap int(K)\neq\emptyset$.
\end{Lem}
\proof If $g(int(K))\cap
int(K)\neq\emptyset$ then it is open, so it contains the interior of an edge $\sigma$. Thus
both $\sigma$ and $g^{-1}\sigma$ are in $int(K)$.
Since $K$
is a finite tree, it contains finitely many open edges. For any such $\sigma\in int(K)$ there
are only finitely many $g\in G$ such that $g^{-1}\sigma\in int(K)$, again because $K$ is a finite
tree and the action of $G$ on edges is free.\qed

As a direct corollary of Lemmas~\ref{lemma:intK} and~\ref{lem:intK2}, we get that  the family
of elements $g\in G$ such that
$gC_n(y)\cap C_n(y)$ could possibly be non-empty is finite and independent of $n$. Therefore, by
Lemma~\ref{lem:almostnoncoll} eventually on $n$, for all $Id\neq g\in G, \ gC_n(y)\cap
C_n(y)=\emptyset$. That is to say, $C_n(y)$ is eventually collapsible. A contradiction.\qed

\section{Optimal maps}
In this section we describe a class of maps, called optimal maps, which provide a useful tool
for computing stretching factors and studying train track maps.

 \begin{Def}[Train track, from~\cite{BestvinaBers}]\label{dtt}
A {\em  pre-train track structure} on a $G$-tree $T$ is a
$G$-invariant equivalence
relation on the set of germs of edges at each vertex of
$T$. Equivalence classes of germs are called {\em gates}.
A {\em train track structure} on a $G$-tree $T$ is a pre-train track structure
with at least two gates at every vertex.
A {\em turn}
is a pair of germs of edges emanating from the same vertex. A turn is
{\em legal} if the two germs belong to different equivalent classes. An
immersed path is legal if it has only legal turns.
  \end{Def}

\begin{Def} Given $A,B\in\mathcal O$ and a PL-map $f:A\to B$, we
  denote by  $A_{\max}(f)$ (or simply $A_{\max}$)
  the subgraph of $A$ consisting on those edges $e$ of
 $A$ for which $S_{f,e}=\Lip(f)$. That is to say, the set of edges
 maximally stretched by $f$.
\end{Def}

Note that $A_{\max}$ is $G$-invariant. We notice that in literature the set $A_{\max}$ is often referred to as {\em tension graph}.

\begin{Def}\label{def:pretrainf} Let $A,B\in\mathcal O$ and $f:A\to B$ be a PL-map. The pre-train track structure
  induced by $f$ on $A$ is defined by
declaring germs of edges to be equivalent if they have the same {\bf non-degenerate} $f$-image.
\end{Def}

\begin{Def}[Optimal map]
Let $A,B\in\mathcal O$. A PL-map $f\colon A\to B$ is {\it not optimal}
at $v$ if $A_{\max}$ has only one gate at $v$ for the pre-train track
structure induced by $f$. Otherwise $f$ is {\it optimal} at $v$. The map
$f$ is {\it optimal} if it is optimal at all vertices.
\end{Def}

\begin{Rem}\label{ropt}
An PL-map $f:A\to B$ is optimal if and only it the pre-train track
structure induced by $f$  is a train track structure on $A_{\max}$.
In particular if
$f:A\to B$ is an optimal map, then at
every vertex $v$ of $A_{\max}$ there is a legal turn
in $A_{\max}$.
\end{Rem}

\begin{Lem} Let $A,B\in\mathcal O$ and let $f\colon A\to B$ be a PL-map.
Then $f$ is optimal at non-free vertices. Equivalently, every non-free vertex has at least two
gates.
 \end{Lem}

 \begin{proof}
    Let $v$ be a non-free vertex of $A$ and let $x$ be an edge emanating
from $v$. If the germ of $x$ is collapsed by $f$ to $f(v)$ then for any $\gamma\in
\operatorname{Stab}(v)$ also $\gamma x$ is collapsed to $f(v)$. By definition such
two germs $x$ and $\gamma x$ are not equivalent in the pre-train track structure induced by
$f$, so $v$ has at least two gates. If
the germ of $x$ is not collapsed to $f(v)$ then by equivariance we have
$f(\gamma x)=\gamma f(x)$, and since $B$ has trivial edge-stabilizers $\gamma f(x)$ is
different from $f(x)$. Therefore $x$ and $\gamma x$ have different non-degenerate images  and
thus are not equivalent. Hence $v$
has again at least two gates.
 \end{proof}

\begin{Lem}\label{lem:optimal}
Let $A,B\in\mathcal O$ and let $f\colon A\to B$ be a PL-map. If $f$ is
not optimal, then there is a PL-map $h\colon A\to B$ such that either
$\Lip(h)\lneq\Lip(f)$ or
$A_{\max}(h)\subsetneq A_{\max}(f)$ (or both).
\end{Lem}
\begin{proof}
Let $v$ be a (free) vertex of $A_{\max}$ where $f$ is not optimal, and let
$e$ be an edge of $A_{\max}$ incident to $v$.

For $t\in[0,l_A(e)]$ let $p_t$ be the point in $e$ at distance $t$ from $v$. Let $f_t$ be the unique
PL-map $A\to B$ such that for $w\in VA$
$$
f_t(w)=
\left\{
\begin{array}{ll}
f(w) & \text{ if } w\neq gv,\ g\in G\\
f_t(gv)=gf(p_t)& g\in G
\end{array}
\right.
$$

For small enough $t$,
if all the edges of  $A_{\max}$ are incident to a point in the orbit of $v$,
 then we obtain that  $\Lip(f_t)\lneq\Lip(f)$;
otherwise we get that $A_{\max}(f_t)$ is obtained from $A_{\max}(f)$
by removing the edges of $A_{\max}$ incident to $v$ and its orbit. We set $h=f_t$.
\end{proof}

\begin{Cor}\label{cor:optimalfromf}
  For any $A,B\in\mathcal O$ there exists an optimal map $h:A\to B$. Moreover, if a PL-map
  $f:A\to B$ is not optimal but minimizes the Lipschitz constant, then there is an optimal map
  $h$ such that
  $A_{\max}(h)\subsetneq A_{\max}(f)$.
\end{Cor}
\proof
By Lemma~\ref{lem:minn} we know $$\Lambda_R(A,B)\leq\inf\{\Lip(h)
\textrm{ s.t. } h:A\to B
\textrm{ is an } \mathcal O\textrm{-map}\}.$$
   By Theorem \ref{thm:finfty} there exists a PL-map $f$ with $\Lip(f)$
  minimal. Among such maps we choose $f$ so that $A_{\max}$ is
  the smallest possible. By Lemma~\ref{lem:optimal} $f$ is optimal.

As for the second claim,
recall that there are finitely many orbits of edges. So if $A_{\max}(h)\subsetneq A_{\max}(f)$,
then the number of orbits of edges in $A_{\max}(h)$ is strictly less than that in
$A_{\max}(f)$. Therefore, given $f$ that minimizes the Lipschitz constant, repeated use of
Lemma~\ref{lem:optimal} gives the desired conclusion.\qed

\begin{Def}
Let $f$ be as in Corollary~\ref{cor:optimalfromf}. By $Opt(f)$ me mean any optimal map $h$ as
given in that corollary.
\end{Def}

\begin{Def}
  Given $A,B\in\mathcal O$ and a PL-map $f:A\to B$, a sub-tree $L\subset A$ is {\em tight} if
  $L\subset A_{\max}$ and $f|_L$ is injective.
\end{Def}
We notice that if $L=\axis_A(g)$ for some $g\in G$, and $L$ is tight, then $f(L)=\axis_B(g)$.

\begin{Thm}\label{thm_optimal_map}
 For any optimal map $f:A\to B$
there is an element $g\in G$ so that its axis in $A$ is tight.
In particular $l_B(g)/l_A(g)=\Lip(f)=\Lambda_R(A,B)$.
\end{Thm}
\begin{proof}
Let $f:A\to B$ be any optimal map.
By Remark~\ref{ropt} every vertex of $A_{\max}$ has a legal
turn. Since pre-train track structures are
$G$-equivariant, we can $G$-equivariantly associate to any edge $e$ of
$A_{\max}$ incident to a vertex $v$, a legal turn $\tau(e)$ at $v$ in
$A_{\max}$, containing $e$. This defines a
successor of $e$ in $A_{\max}$.
Starting from an edge
$e_0$ in $A_{\max}$, the path obtained by concatenating successors,
defines an embedded legal half-line, which eventually becomes periodic
because $A$ has finitely many orbits of edges. The period $g\in G$ has the
requested properties.
\end{proof}


\section{Folding paths and geodesics}

\subsection{Local folds}
For this sub-section we fix $A,B\in \mathcal O$, and a PL-map $f:A\to B$.

\begin{Def}[Isometric folding relations]
  For any $v\in VA$, $t\in\R$ and pair $\tau$ of edges $\tau=(\alpha,\beta)$ with
  $S_\alpha(f)=S_\beta(f)$, and emanating from $v$ such that $f(\alpha)$ and $f(\beta)$ agree on some non-trivial segment.
  We define an equivalence relation $\sim_{\tau,t}$ on $A$ as follows. First we declare $x\in
  \alpha$ and   $y\in \beta$ to be equivalent if $d(x,v)=d(y,v)\leq t$
 and, after the isometric identification of $[v,x]$ and $[v,y]$, we have
 $f|_{[v,x]}=f|_{[v,y]}$. Then we extend this relation to the orbit of $\tau$ by equivariance.
\end{Def}

\begin{Def}
  Given $v,\tau$ as above, and $t\in\R$ we define $A_{\tau,t}=A\sim_{\tau,t}$ equipped with the
  metric making the quotient map $q_{\tau,t}:A\to A_{\tau,t}$ a local isometry. The map $f$
  splits as $$\xymatrix{A\ar[r]^f\ar[d]_{q_{\tau,t}}&B\\ A_{\tau,t}\ar[ur]_{f_{\tau,t}}}$$
We say that $A_t$ is obtained by (equivariantly) folding $\tau$.
If $\tau$ is understood we shall abuse notation and suppress the subscript $\tau$.
\end{Def}

By definition, a fold depends on how $f$ overlaps edges.
When necessary, we will say that a fold is {\bf directed by $f$} to emphasize this fact.

\begin{Lem}\label{lem:tensionfold}
In the present setting, for any $t$ we have that  either $\Lip(PL(f_t))<\Lip(f)$ or
$A_{t\, \max}(PL(f_t))\subseteq q_t(A_{\max}(f))$.
\end{Lem}
\proof
Fix $t$. Let  $\sigma$ be the $f_t$-pullback simplicial structure on $A_t$. We write
$A_t^\sigma$ and $A_t$ to distinguish between $\sigma$ and the original simplicial structure of $A_t$.
Note that $f_t$ is then a $\sigma$-PL map. As $q_t$ is a local isometry, for any
edge of $\sigma$ we have $S_{e}(f_t)=S_{q^{-1}_t(e)}(f)$ . In particular, $\Lip(f_t)=\Lip(f)$
and $A^\sigma_{t\,\max}(f_t)=q_t(A_{\max}(f))$.
Now the edge-stretching factors of $PL(f_t)$ are less than or equal to those of $f_t$. Hence $\Lip(PL(f_t))<\Lip(f_t)$ or $A_{t\, \max}(PL(f_t))\subseteq A^\sigma_{t\, \max}(f_t)$.\qed

The following lemma will be useful in the study of train track maps. For $\Phi\in Aut(G,\mathcal
O)$, $\Phi(A)$ and $A$ are the same metric tree with different $G$-action. So
$\Phi(A)_t=\Phi(A_t)$ and we use the same symbol
$q_t$ to denote the quotient map from $\Phi(A)\to\Phi(A_t)$.
We have the following commutative diagram which defines the map $h_t$ (Figure~\ref{fig:ht}).
\begin{figure}[h]
  \centering
$$
\xymatrix{A\ar[r]^f\ar[d]_{q_t}&\Phi(A)\ar[d]^{q_t}\\
A_t\ar[ur]_{f_t}\ar[r]_{h_t}&\Phi(A_t)}
$$

  \caption{The quotient map $h_t$}
  \label{fig:ht}
\end{figure}

\begin{Lem}\label{lem:square}
  Let $\Phi\in\Aut(G,\mathcal O)$ and suppose that $f:A\to \Phi(A)$ is a PL-map such that
$$\Lip(f)=\min\{\Lip(h),\  h:X\to \Phi(X), X\in \mathcal O\}.$$
Let $A_t$ be the tree obtained by perform a local fold directed by $f$. Then
$$\Lip(PL(h_t))=\Lip(PL(f_t))=\Lip(f)$$
and
$$
A_{t\, \max}(PL(h_t)) \subseteq q_t(A_{\max}(f)).
$$
Moreover, if $PL(h_t)$ is not optimal, then
$$
A_{t\, \max}(Opt(PL(h_t)))\subsetneq A_{t\, \max}(PL(h_t)).
$$
\end{Lem}
\proof
Since $q_t$ is a local isometry, then $\Lip(h_t)=\Lip(f_t)=\Lip(f)$. Passing to $PL$ does not
increase the Lipschitz constants, and
by hypothesis $\Lip(f)$ is minimal, thus $$\Lip(PL(h_t))=\Lip(PL(f_t))=\Lip(f).$$ Hence, by
Lemma~\ref{lem:tensionfold} we have that $A_{t\, \max}(PL(f_t))\subseteq
q_t(A_{\max}(f))$.  Now, note that $PL(q_t\circ PL(f_t))=PL(h_t)$. Therefore, since $q_t$ is a
local isometry, we have $A_{t\,\max}(PL(h_t))\subseteq A_{t\ \max}(PL(f_t))$. Whence the second
claim. The last claim is a direct consequence of Corollary~\ref{cor:optimalfromf}.
\qed

\subsection{Folding paths}
Now we describe paths joining any two points of $\mathcal O$ which are geodesics w.r.t. the
metric $d_R$. The procedure is exactly that used in~\cite{FrancavigliaMartino} in the case of
free groups.

First we restrict attention to the special situation where $f:A\to B$ is a map so that
$A_{\max}(f)=A$ and $\Lip(f)=1$.

For a complete simple fold we mean the path obtained by equivariantly folding two edges $\alpha$ and $\beta$ as much as
possible. That is to say, the path $[0,m]\to \mathcal O$
$$t\mapsto A_{\tau,t}$$
where  $\tau=(\alpha,\beta)$, $M=\min\{L_A(\alpha),L_A(\beta)\}$,  $\sim_{\tau,M}$ is not trivial, and
$m$ is the minimum $t$ so that $\sim_{\tau,t}=\sim_{\tau,M}$ .

\begin{Prop}\label{prop:foldingpath}
  Let $A,B\in \mathcal O$ and $f:A\to B$ a PL-map such that $A_{\max}(f)=A$ and
  $\Lip(f)=1$. Then there exists a path from $A$ to $B$ which is a concatenation of complete
  simple folds directed by $f$.
\end{Prop}
\proof Let $\sigma$ be the simplicial structure induced on $A$ by $f$, so that $f$ maps
$\sigma$-edges to edges. Note that by
definition of $\sigma$ if the initial segments of two edges of $\sigma$ have the same $f$-image
then the two edges have the same $f$-image. On the other hand, if $f$ admits no simple folds,
then $f$ is a $G$-equivariant isometry from $A$ to $B$, and hence $A=B$. Otherwise,
let $A_t$ be a tree obtained by a complete simple fold.

 Then $f$
splits as
$$\xymatrix{A\ar[r]^f\ar[d]_{q_{t}}&B\\ A_{t}\ar[ur]_{f_{t}}}$$
and $\sigma$ induces a simplicial structure on $A_t$ with fewer orbits of edges. Induction
completes the proof.\qed

Now we come back to the general case.

\begin{Def}[Isometric folding paths]
  Let $A,B\in\mathcal O$ and $f:A\to B$ be a PL-map such that $A_{\max}(f)=A$.  A {\bf
    isometric folding path} from $A$ and $B$, and directed by $f$, is a path obtained as follows.
  \begin{itemize}
\item  First,  rescale the metric on $A$ so that $\Lip(f)=1$ and call that point $A_0$. Note
  that $A_0=\Lip(f)A$.
\item Then, consider a path  $\hat\gamma(t)=A_t$ from $A_0$ to $B$ given by
  Proposition~\ref{prop:foldingpath}, parametrized by $t\in[0,1]$.
\item Finally, rescale $\hat\gamma$ by $\gamma(t)=\hat\gamma(t)/\Lip(f)^{(1-t)}$.
  \end{itemize}
\end{Def}

\begin{Thm}
For any $A,B\in\mathcal{PO}$ there is a $d_R$-geodesic in $\mathcal{PO}$ from $A$ to $B$.
\end{Thm}
\proof
We use the following characterization of (unparameterized) geodesics: an oriented path $\gamma$ in $\mathcal{PO}$ is a
$d_R$-geodesics if and only if there is a $g\in G$ hyperbolic so that
$$\Lambda_R(C,D)=\frac{l_D(g)}{l_C(g)}$$
for any $C<D\in\gamma$.

We now describe a path in $\mathcal O$ which projects to a $d_R$-geodesics in $\mathcal{PO}$.

Let $f:A\to B$ be an optimal map and let $g\in G$ be an element with tight axis in $A$.
(Theorem~\ref{thm_optimal_map}). First we equivariantly rescale all the edges of $A$ as follows
$$e\mapsto S_e(f)e$$
obtaining a tree $A_0$. Call $p$ the projection map $p:A\to A_0$ clearly $f$ splits as
$$\xymatrix{A\ar[r]^f\ar[d]_{p}&B\\ A_{0}\ar[ur]_{f_{0}}}$$
Note that if $S_e(f)=0$ for some edge $e$, than $f$ collapses $e$, so the  tree $A_0$ is
actually in $\mathcal O$.  Moreover, the map   $f_0:A_0\to B$ is PL and has  the property  that
$A_{0\,\max}(f_0)=A_0$ and $\Lip(f_0)=1$. Finally, note that if $S_e(f)=0$ then $e\notin
A_{\max}(f)$, therefore $p$ restricts to a homeomorphism from $A_{\max}(f)$ to its image. In
particular, this implies that $g$ has tight axis in $A_0$.

We do this operation continuously so that we have an oriented path $\gamma_1$ from $A$ to a $A_0$.
It is clear that for any two points $C,D$ in $\gamma_1$
$$\Lambda_R(C,D)=\frac{l_D(g)}{l_C(g)}.$$

Let $\gamma_2$ be an isometric folding path from $A_0$ to $B$ directed by $f_0$.
Define $\gamma$ to
be the concatenation of $\gamma_1$ and $\gamma_2$.

Since $g$ has a tight axis in $A_0$, by definition $f_0$ is
injective on $\axis_{A_0}(g)$. Therefore, $\axis_{A_0}(g)$ is never folded in $\gamma_2$. In
particular, $l_{\gamma_2(t)}(g)$ is constant. Since $\gamma_2$ is an isometric folding path, nothing is
is stretched by a factor $\geq 1$, hence $g$
realizes $\Lambda_R(C,D)=1$ for any two ordered points in $\gamma_2$.

This implies that for any $C<D$ in $\gamma$
$$\Lambda_R(C,D)=\frac{l_D(g)}{l_C(g)}.$$

Finally, note that the above condition is scale invariant, in the sense that if
$$\Lambda_R(C,D)=\frac{l_D(g)}{l_C(g)}$$
then
$$\Lambda_R(\lambda C,\mu D)=\frac{l_{\mu D}(g)}{l_{\lambda C}(g)}$$
for any $\lambda,\mu>0$.

Therefore the projection of $\gamma$ to the set of co-volume one elements of $\mathcal O$ is an
(unparameterized) $d_R$-geodesic.\qed


\section{Train tracks}\label{sectionTT}
 In this section we prove that any irreducible automorphism in $\Aut(G,\mathcal O)$
 is represented
  by a train track map (see below the definitions). We follow the approach {\em \`a la Bers}
  of M. Bestvina (\cite{BestvinaBers}). The following arguments are restatement of those for the case of free
  groups. In fact, the proofs do not require adjustment due to the fact that we are allowing
  non trivial stabilizer for vertices, and one could just say that the theory of train tracks
  for free groups passes to the case of free products without any substantial change. We refer
  to~\cite{BFH1,BFH2,BFH3,BFH4,BestvinaHandel} for the train track theory for free groups.
However,
  there are many facts that are well-known for free groups, at least to the
  experts, but for which there is no reference in literature.  We take the occasion of the
  present discussion on train tracks for free product to give explicit statements and metric
  proofs  of some of these facts, as for instance the relations between the minimal displaced
  set and the set of train tracks. (See also~\cite{HM,MP}).

\subsection{Irreducibility and minimal displacement trichotomy}
Let $\Phi\in\Aut(G,\mathcal O)$. It acts on $\mathcal
  O$ by changing the marking.
Define $$\lambda_\Phi=\inf_{X\in\mathcal O}\Lambda_R(X,\Phi(X)).$$

Note that since both $X$ and $\Phi(X)$ have the same volume, the number $\lambda_{\Phi}$ cannot be smaller than one.
Therefore, there are three cases: $\Phi$ is {\em elliptic}, if
$\log \lambda_\Phi$ is zero and the infimum is attained;
{\em parabolic}, if the infimum is not attained; {\em hyperbolic} if $\log \lambda_\Phi$ is positive and
attained.

For $T\in \mathcal O$ we say that a Lipschitz surjective map $f:T\to
T$ represents $\Phi$ if for any
$g\in G$ and $t\in T$ we have $f(gt)=\Phi(g)(f(t))$. (In other words, if
it is an $\mathcal O$-map from $T$ to $\Phi(T)$.)

\begin{Def}
We say $\Phi\in\Aut(G,\mathcal O)$ is $\mathcal O$-irreducible (or simply irreducible for
short) if for any $T\in O$ and for any $f:T\to T$
representing $\Phi$, if $W\subset T$ is a proper $f$-invariant $G$-subgraph then
$\quot GW$ is a union of trees each of which contains at most one non-free vertex.
\end{Def}

This is related to an algebraic definition of irreducibility as follows. Suppose that $G$ can be written as a free product,
$G=G_1*G_2* \ldots G_k*G_{\infty}$, where we allow the possibility that $G_{\infty}$ is trivial.
Then we say that the set $\mathcal{G}=\{ [G_i] \ : \ 1 \leq i \leq k \}$ is a {\em free factor system} for $G$,
where $[G_i]= \{ g G_i g^{-1} \ : \ g \in G \}$ is the set of conjugates of $G_i$. Given two free factor systems
$\mathcal{G}=  \{ [G_i] \ : \ 1 \leq i \leq k \}$ and $\mathcal{H} = \{ [H_j] \ : \ 1 \leq j \leq m \}$, we write
$\mathcal{G} \sqsubseteq \mathcal{H} $ if for each $i$ there exists a $j$ such that $G_i \leq g H_j g^{-1}$ for some
$g \in G$. We write $\mathcal{G} \sqsubset \mathcal{H} $ if one of the previous inclusions is strict. We also say that
$\mathcal{G}=\{ [G_i] \ : \ 1 \leq i \leq k \}$ is proper if $\mathcal{G} \sqsubset \{ [G] \}$.

We say that $\mathcal{G}=\{ [G_i] \ : \ 1 \leq i \leq k \}$  is $\Phi$ invariant for some
$\Phi \in Out(G)$ if for each $i$, $\Phi([G_i]) = [G_j]$ for some $j$. We shall restrict our attention to those free factor systems
$\mathcal{G}$ such that $\{ [H] \} \sqsubseteq \mathcal{G}$ whenever $H$ is a free factor which is {\em not} a free group. In particular,
this means that $G=G_1*G_2* \ldots G_k*G_{\infty}$, and $G_{\infty} \cong F_k$ for some free group $F_k$.  Associated to such a free factor system
$\mathcal{G}=\{ [G_i] \ : \ 1 \leq i \leq k \}$ we have the space of trees $\mathcal{O}=\mathcal{O}(G,(G_i)_{i=1}^p,F_k)$ and any (outer)
automorphism of $G$ leaving $\mathcal{G}$ invariant will act on $\mathcal{O}$ in the usual way.

\begin{Def}
Let $\mathcal{G}$ be a free factor system of $G$ as above and suppose it is $\Phi$ invariant for some $\Phi \in Out(G)$.
Then $\Phi$ is called irreducible relative to $\mathcal G$ if $\mathcal{G}$ is a maximal (under $\sqsubseteq$) proper, $\Phi$-invariant free factor system.
\end{Def}

The following is clear.

\begin{Lem}
Suppose $\mathcal{G}$ is a free factor system of $G$ with associated space of trees $\mathcal
O$, and further suppose that $\mathcal G$ is $\Phi$-invariant.
Then $\Phi$ is irreducible relative to $\mathcal G$ if and only if $\Phi$ is $\mathcal O$-irreducible.
\end{Lem}

We prove now  that irreducible automorphisms are hyperbolic.

\begin{Thm}
For any irreducible $\Phi\in\Aut(G,\mathcal O)$, $\lambda_\Phi=\inf_{X\in\mathcal O}\Lambda_R(X,\Phi(X))$ is a minimum and obtained for some $X \in {\mathcal O}$.
\end{Thm}
\proof We essentially follow the proof in \cite{BestvinaBers}, noting that the technical difficulties arise due to the fact that our space is not locally compact and
that our action is not proper. We shall utilise the Sausages Lemma result, Theorem~\ref{sausages}, which is proved in the subsequent section, but whose proof is
independent on the results in this section.

In order to proceed, we demonstrate the contrapositive, that an automorphism,
$\Phi\in\Aut(G,\mathcal O)$, for which  $\lambda_\Phi=\inf_{X\in\mathcal O}\Lambda_R(X,\Phi(X))$ is not a minimum is reducible.
So we suppose that  $\Gamma_k \in \mathcal O$ is a minimising sequence for $\Phi$. That is, $\lim_{k \to \infty} d_R(\Gamma_k, \Phi\Gamma_k) \to \log \lambda_\Phi$.

We notice that $d_R(\Gamma,\Phi\Gamma)$ is scale-invariant as a function of $\Gamma$, and hence
descends to a function on $\mathcal{PO}$. For the remaining part of this proof, we work with the
co-volume one slice of $\mathcal O$, which we still denote by $\mathcal O$ for simplicity of notation.

Our first step is to show that the trees $\Gamma_k$ cannot stay in the `thick' part of
$\mathcal O$. The $\epsilon$-thick part of (the co-volume one slice of) $\mathcal O$ consists of all trees $X \in \mathcal O$
such that  $l_X(g) \geq \epsilon$ for all hyperbolic $g \in G$. Note that the $\epsilon$-thick
part of $\mathcal O$ is co-compact  for any $\epsilon >0$.

More precisely, we wish to show that only finitely many of the $\Gamma_k$ lie in the $\epsilon$-thick part of $\mathcal O$ for any $\epsilon > 0$. For suppose not, then passing to a subsequence we may assume that
all $\Gamma_k$ belong to the $\epsilon$-thick part and then, again by taking subsequences and
invoking co-compactness, we may find $\Psi_k \in  \Aut(G,\mathcal O)$ such that
$\Psi_k(\Gamma_k)$ converges to some $\Gamma_{\infty}$ which is again in the $\epsilon$-thick
part of $\mathcal O$.
Hence $d_R(\Gamma_{\infty}, \Psi_k \Phi \Psi_k^{-1} \Gamma_{\infty}) \to \log \lambda_\Phi$.

However, note that as we are dealing with simplicial trees, the translation lengths of the elements in a given tree form a discrete set. In fact, the set
$\{ l_{ \Psi(\Gamma_{\infty})}(g) \ : \ g \in G \}$ is the {\em same} discrete set for any $\Psi \in Aut(G, \mathcal O)$.  Moreover, by Theorem~\ref{sausages}, $d_R(\Gamma_{\infty}, \Psi_k \Phi \Psi_k^{-1} \Gamma_{\infty})$ are given by the quotient of the translation lengths of candidates, and there are only finitely many possible lengths of candidates in $\Gamma_{\infty}$ (even though there will, in general, be infinitely many candidates) and therefore the distances $d_R(\Gamma_{\infty}, \Psi_k \Phi \Psi_k^{-1} \Gamma_{\infty})$ also form a discrete set. Hence, there must exist some $k$ (in fact infinitely many) such that $d_R(\Gamma_{\infty}, \Psi_k \Phi \Psi_k^{-1} \Gamma_{\infty})=\log \lambda_\Phi$, whence we obtain that $\lambda_\Phi=\Lambda_R(X,\Phi(X))$, where $X=\Psi_k^{-1} \Gamma_{\infty}$.

Now for any $\Gamma \in \mathcal O$, we let $\Gamma^{\epsilon}$ be the sub-forest obtained as
the union of all the hyperbolic axes of elements of $G$ whose translation length is less than
$\epsilon$, along with all the vertices. Since there are only finitely many graphs of groups
arising from $\mathcal O$, each of which is finite, there exists an $\epsilon$ such that
for all $\Gamma$,  $\Gamma^{\epsilon}$ is a proper sub-forest of $\Gamma$ (we remind that we are now
working with co-volume one trees). Call this $\epsilon_0$.

Also notice that each such sub-forest is a $G$-invariant subgraph, and hence there is a bound
on the length of any proper chain of such sub-forests. Call this number $B$.

Now let $\epsilon_i=\epsilon_0/ (\lambda_\Phi+1)^i$. Choose $\Gamma=\Gamma_k$ as above such that $d_R(\Gamma, \Phi \Gamma) < \log (\lambda_\Phi
+1)$ and $\Gamma$ not in the $\epsilon_B$-thick part (so $\Gamma^{\epsilon_B}$ is non-trivial). Now,

$$
\Gamma\neq\Gamma^{\epsilon_0} \supseteq \Gamma^{\epsilon_1} \supseteq \ldots \supseteq \Gamma^{\epsilon_B}$$
is a chain of $(B+1)$ non-trivial sub-forests of $\Gamma$. Therefore they cannot all be distinct. However, if $f: \Gamma \to \Phi \Gamma$ is any optimal map, then $f$ must send $\Gamma^{\delta_i}$ into $\Gamma^{\delta_{i-1}}$. Hence, we must have an $f$-invariant subgraph of $\Gamma$ which is non-trivial, and hence $\Phi$ is reducible.
\qed

Hence from now on we will use the fact that all our  $\mathcal O$-irreducible elements of $\Aut(G,\mathcal O)$ are hyperbolic.

\subsection{Minimally displaced points}
For an irreducible automorphism $\Phi$ we introduce the set of minimally displaced points,
which plays the role of a ``metric'' axis for $\Phi$. We will show later that this coincides
with the set of train tracks.

 \begin{Def}[Minimal displaced set]
      Let $\Phi$ be an $\mathcal O$- irreducible
  element of $\Aut(G,\mathcal O)$. We define the minimal displaced set of $\Phi$ by
$$\mathcal M(\Phi)=\{T\in \mathcal O\ :\ \Lambda_R(T,\Phi (T))=\lambda_\Phi\}$$
 \end{Def}
 \begin{Thm}\label{lem:good}   Let $\Phi$ be an $\mathcal O$-irreducible
   element of $\Aut(G,\mathcal O)$. Then,
if $T\in\mathcal M(\Phi)$ and $f:T\to \Phi(T)$ is an optimal map, we have
$$T_{\max}(f)=T.$$
 \end{Thm}
\proof We consider $f$ as either an $\mathcal O$-map from $T\to \Phi(T)$ or a map $f:T\to T$
representing $\Phi$, without distinction. If $T_{\max}$ is $f$-invariant we are done because, since by
Theorem~\ref{thm_optimal_map} $T_{\max}$ contains the axis of some hyperbolic element,
irreducibility implies $T=T_{\max}$.

In the subsequent argument we shall perform small
perturbations on $T$ by changing edge-lengths. The map $f$ will induce maps on these new trees
which are the same as $f$ set-wise. Formally we have many different pairs of trees and
associated maps, but that we still call $(T,f)$.

Suppose that $T_{\max}$ is not $f$-invariant. Then there is an edge $e$ in $T_{\max}$ whose image
contains and edge $a$ which is not in $T_{\max}$. We shrink $a$ by a small amount.
If the perturbation is small enough, $\Lip(f)$ is not increased,
and since $\Lip(f)=\lambda_\Phi$ this
remains true after the perturbation. Therefore, $e$ is no longer in $T_{\max}$,
$a$ is still not in $T_{\max}$, and $T_{\max}$ must contain some other edge $b$ with $S_b(f)=\lambda_\Phi$.

Note that after perturbation $f$ might no longer be optimal. However, by Corollary~\ref{cor:optimalfromf} there is an optimal map $h:T\to\Phi(T)$ with
$$T_{\max}(h)\subsetneq T_{\max}(f).$$ Since $\Lip(f)=\lambda_\Phi$, then also
$\Lip(h)=\lambda_\Phi$.
If $T_{\max}(h)$ is not $f$-invariant, we repeat this argument recursively.
After finitely many steps we must end obtaining a finite sequence of maps $h_i$ with the
properties that $\Lip(h_i)=\lambda_\Phi$ and
$$\emptyset\neq T_{\max}(h_0)\subsetneq T_{\max}(h_1)\subsetneq \dots \subsetneq T_{\max}(f)\subsetneq T.$$

Since we stopped, $T_{\max}(h_0)$ is $f$-invariant. Therefore by irreducibility
$T_{\max}(h_0)=T$, and the above
condition implies $T=T_{\max}(f)$.
\qed

\medskip

\begin{Rem}
In the proof of Theorem~\ref{lem:good} we showed the following fact which needs no
assumption on reducibility, and may be of
independent interest: If $(T,f)$ are so that first, $T$ locally weak minimizes $d(T,\Phi T)$,
and second $T_{\max}$ is locally minimal, then $T_{\max}$ is $f$-invariant. 
\end{Rem}

\begin{Lem}\label{lem:extract}
  Suppose $T\in\mathcal M(\Phi)$ and suppose that $f:T\to T$ is a Lipschitz map with
  $\Lip(f)=\lambda_\Phi$. Then $f$ is optimal.
\end{Lem}
\proof
First, consider $PL(f)$. Since $\Lip(PL(f))\leq\Lip(f)$ and $\Lip(f)$ is minimal, we have
$\Lip(PL(f))=\lambda_\Phi$. Moreover, combining
Lemma~\ref{lem:square} and
Theorem~\ref{lem:good}, we get that $PL(f)$ is optimal and $T_{\max}(PL(f))=T$.

Since
the maps $PL(f)$  have the property that $S_e(PL(f))\leq S_e(f)$ with inequality being strict
at some edge only if $f$ is not $PL$, it follows that $f=PL(f)$. \qed

 \begin{Thm}[$\mathcal M(\Phi)$ is fold-invariant.]\label{thm:foldinv}
      Let $\Phi$ be an $\mathcal O$- irreducible
  element of $\Aut(G,\mathcal O)$. Then, the set $\mathcal M(\Phi)$ is invariant
    under folding directed by optimal maps.

More precisely,  if $T\in\mathcal M(\Phi)$ and  $f:T\to \Phi(T)$ is an optimal map, and
 if $T_t$ is an isometric
    folding path from $T\to \Phi(T)$ directed by $f$, then  we have:
    \begin{enumerate}[a)]
    \item $T_t\in \mathcal M(\Phi)$.
    \item The quotient maps
    $h_t:T_t\to\Phi(T_t)$, defined by the diagram in Figure~\ref{fig:ht}, are optimal.
    \end{enumerate}
 In particular, any local fold directed by $f$ stays in $\mathcal M(\Phi)$.
 \end{Thm}
\proof Claim $a)$ is a  direct consequence of Lemma~\ref{lem:square}. Claim $b)$ follows from
Lemma~\ref{lem:extract} because $\Lip(h_t)=\Lip(f)=\lambda_\Phi$.
\qed

\subsection{Train track maps}
Train track maps can be defined via topological properties as well as metric properties. In
this section we relate the two point of view.  

Recall that we defined pre-train track and train track structures in
Definition~\ref{dtt}, and that in our notation a train track structure is required to have {\bf
  at least two gates} at every vertex.
  \begin{Def}[Train track map]\label{Defttm}
    A PL-map $f:T\to T$ representing $\Phi$ is a {\em train track map}
    if there is a train track structure on $T$ so that
    \begin{enumerate}[1)]
    \item $f$ maps edges to legal paths (in particular, $f$ does not collapse edges);
    \item If $f(v)$ is a vertex, then $f$ maps inequivalent germs at $v$ to inequivalent germs
      at $f(v)$.
    \end{enumerate}
  \end{Def}

Here a some remark is needed. First, we note that a part the PL requirement, this
definition is topological and does not involves the metric on $T$. In fact if $f$ is train
track and we change the metric on edges of $T$, then up to re-PL-ize $f$ it remains train
track. For these reasons we have to distinguish  between (topological) train track maps and
(metric) optimal train track maps.

The second remark on the definition of train track map is that,  given an optimal map $f$ representing $\Phi$, we can consider
two pre-train track structures, namely that given by $f$ and that generated by the iterates
$f^k$. We denote the two structures in the following way
$$\sim_f \qquad\text{ and }\qquad \langle\sim_{f^{k}}\rangle$$
So, two germs are $\sim_f$-equivalent if they are identified by $f$, they are
$\sim_{f^k}$-equivalent if they are identified by $f^k$ and they are $\langle
\sim_{f^k}\rangle$-equivalent if they are identified by some power of $f$.

In particular one may ask if $f$ is a train track for $\sim_f$ or for $\simfk$.

\begin{Lem}\label{lem:lemmaX}
  Suppose $f:T\to T$ is a PL-map representing $\Phi\in\Aut(G,\mathcal O)$.  If $f$ is a train
  track map for $\sim$, then $\sim\supseteq\simfk$. In particular, if $f$ is a train track map
  for $\sim_f$, then $\sim_f=\simfk$.
\end{Lem}
\proof The last claim follows from the first because $\sim_f\subseteq \simfk$ by definition. 
Suppose $\tau=(e_1,e_2)$ is a turn and suppose that $e_1$ and
$e_2$ are in the same gate for $\simfk$. Then there is some $k$ so that
$f^k(e_1)=f^k(e_2)$, choose the first $k\geq 1$ so that this happens. Either $f^{k-1}(\tau)$ is
contained in an edge or it is a turn. The first case is not allowed since $f$ is a $\sim$-train track
map. Therefore $f^{k-1}(\tau)$ is an turn. Since
$f$ identifies the two germs of $f^{k-1}(\tau)$, by Condition $2)$ of Definition~\ref{Defttm},
$f^{k-1}(\tau)$ is $\sim$-illegal. Since $f$ is a train track map, turns that are pre-images of
illegal turns are illegal. It follows that $\tau$ is illegal, that is to say,
$e_1\sim e_2$.\qed

\begin{Cor}\label{CorX}
 Is $f$ is a train
  track map for some $\sim$, then it is train track for $\simfk$.
\end{Cor}
\proof Condition $2)$ of Definition~\ref{Defttm} is automatically satisfied for
$\simfk$. Lemma~\ref{lem:lemmaX} tells us that $\sim$-legal implies $\simfk$-legal. Thus also
condition $1)$ of Definition~\ref{Defttm} is satisfied.\qed

Note that a priori $f$ could be train track for $\simfk$ but not for $\sim_f$.

\begin{Lem}[Topological characterization of train track maps]\label{NuovoLemma}
If $\Phi$ is  irreducible, then for a map $f$ representing $\Phi$, to be a train track map
is equivalent to the condition that  there is a hyperbolic $g\in G$  with axis $L$ so that  
$f^k|_L$ is injective $\forall k\in\mathbb N$.
\end{Lem}
\proof
Let $L$ be as in the hypothesis. The iterate images of $L$ form a proper $f$-invariant
sub-graph of $T$ containing the axis of a hyperbolic element. Since $\Phi$ is irreducible, such
sub-graph is the whole $T$. The pre-train track structure $\simfk$ is a train track structure and
it is readily checked that $f$ satisfies Conditions $1)$ and $2)$ of Definition~\ref{Defttm}
 with respect to $\simfk$.

On the other hand, suppose that $f$ is a train track map. As train track structures have at
least two gates at every vertex,  there is a hyperbolic  element $g$ with legal axis $L$, and
Conditions $1)$ and $2)$ imply that this remains true under $f$-iterations. 

By Corollary~\ref{CorX} it is a $\simfk$-train track map, and $\simfk$-legality of $L$ implies
injectivity of $f^k|_L$. 
\qed

\begin{Def}[Train track bundle]
  Let $\Phi$ be an $\mathcal O$-irreducible element of $\Aut(G,\mathcal O)$. We
  define the train track bundle as
$$TT(\Phi)=\{T\in\mathcal O\ : \exists \text{ an optimal train track map } f:T\to T \text{
  representing }\Phi\}$$
\end{Def}

We notice that the Axis bundle $\mathcal A_\Phi$ of $\Phi$ is defined in~\cite{HM} as the closure of the union of all the sets $TT({\Phi^k})$.
\begin{Def}[Strict train tracks]
  Let $\Phi$ be an $\mathcal O$-irreducible element of $\Aut(G,\mathcal O)$. We
  define the strict train track bundle as
$$TT_0(\Phi)=\{T\in\mathcal O\ :\ \exists f:T\to \Phi T \text{ optimal which is train track for } \sim_f\}$$
\end{Def}

\begin{Lem}
  Suppose $f:T\to T$ is a train track map representing $\Phi$. Then there is a rescaling of
  edges of $T$ such that every edge is stretched the same , and hence $f$ becomes optimal. In
  particular, if $f$ is 
  train track, the simplex of $T$ contains a point $T'\in TT_0(\Phi)$.
\end{Lem}
\proof Since $f$ is a train track map, id does not collapse any edge. Therefore, starting from
less stretched edges we can shrink the length of any edge so that every edge is stretched the
same. Let $T'$ be the point in the simplex of $T$ obtained in this way. Clearly $T'_{\max}=T'$
and the $PL$-ization of $f$ gives a train track map $f'$. Since $f'$ is train track and
$T'_{\max}=T'$, then $f'$ is optimal and $T'\in TT_0(\Phi)$.\qed

Now, we want to prove that minimally displaced points and train tracks coincides.  As a first
observation we have.
\begin{Lem}
  For any $\Phi\in\Aut(G,\mathcal O)$, we have $TT_0(\Phi)\subseteq TT(\Phi)\subseteq \mathcal M(\Phi)$
\end{Lem}
\proof
If $T\in TT(\Phi)$ and $f$ is an optimal train track map, then there is $g$ such that
$l_{T}(\Phi^n(g))=\Lip(f)^nl_T(g)$. On the other hand, if $Q\in\mathcal M(\Phi)$ then
$l_Q(\Phi^n(g))\leq \lambda_\Phi^nl_Q(g)$. Therefore we have
$$\left(\frac{\Lip(f)}{\lambda_\Phi}\right)^n
\frac{l_T(g)}{\lambda_\Phi l_Q(g)}<\Lambda_R(\Phi^nT,\Phi^nQ)=\Lambda_R(T,Q)<\infty$$
which implies $\Lip(f)=\lambda_\Phi$.\qed

\begin{Thm}\label{thm:A_0=M}
  Let $\Phi$ be an irreducible
  element of $\Aut(G,\mathcal O)$. Then,
  $\emptyset\neq TT_0(\Phi)$ is dense in $\mathcal M(\Phi)$ with respect to the
  simplicial topology.
\end{Thm}
\proof
Since $\Phi$ is hyperbolic there exists an element $T\in\mathcal O$ that realizes
$\lambda_\Phi$. Let $f:T\to \Phi(T)$ be an optimal map, which exists by
Corollary~\ref{cor:optimalfromf}. By Theorem~\ref{lem:good} we know that $T_{\max}(f)=T$.

We consider the pre-train track structure $\sim_f$ induced  by $f$ on $T$.
As
$f$ is optimal the pre-train track structure is a train track structure (no one-gate vertex).
Now, say that a vertex of $T$ is {\em foldable} if it contains at least a gate with two elements.

By Theorem~\ref{thm:foldinv}, up to perturbing $T$ by as small an amount as required via a finite number of
equivariant folds, we may assume that any
foldable vertex has valence exactly $3$. In particular any foldable vertex has exactly two gates.

Moreover, again by Theorem~\ref{thm:foldinv} we may assume that $T$ locally maximizes the number
of orbits of foldable vertices.

We claim that in this situation $f$ is a train track map with respect to $\sim_f$. First, we check Condition $1)$ of Definition~\ref{Defttm}.

Suppose that an edge $e$ of $T_{\max}$ has
illegal image. Then it passes through an illegal turn  $\tau$. We
equivariantly fold $\tau$ by a small amount $t$.
The result is a new tree $T_t$ and an induced map
$h_t$. By Theorem~\ref{thm:foldinv} $T_t\in\mathcal M(\Phi)$ and
  $h_t$ is optimal. But $e\notin T_{t\ \max}(h_t)$, so
 $T_{t\ \max}\neq T_t$, which is impossible by  Theorem~\ref{lem:good}.

Now we check Condition $2)$.
By definition of our pre-train track structure, inequivalent germs are mapped
to different germs. Now, suppose that a legal turn $\eta$ at a
vertex $v$ is mapped to an illegal
turn $\tau$ at a vertex $w$. We equivariantly fold $\tau$ by a small
amount getting a tree $T_t$ and a map $h_t$. Theorem~\ref{thm:foldinv} guarantees that
$T_t\in\mathcal M(\Phi)$ and $h_t$ is optimal. Now, $\eta$
became foldable. By optimality there are no one-gate vertices. Thus
we increased the number of foldable vertices in contradiction with our assumption of
maximality.
\qed

In fact, more is true.
\begin{Thm}\label{thm:A=M}
  Let $\Phi$ be an irreducible
   element of $\Aut(G,\mathcal O)$. Then,
  $TT(\Phi)=\mathcal M(\Phi)$.
\end{Thm}
\proof
What we are going to prove is that if $f$ is an optimal map representing $\Phi$, then it is a
train track map for $\simfk$. We need a couple of lemmas.

\begin{Lem}\label{lem:lemmaY}
  Suppose $f:T\to T$ is a PL-map representing $\Phi\in\Aut(G,\mathcal O)$.
 If $f$ is a train track map for $\sim_f$, then $f^k$, which represents $\Phi^k$, is a train track map for $\sim_{f^k}$.
\end{Lem}
\proof
By Lemma~\ref{lem:lemmaX} $\sim_f=\simfk$, whence $\sim_f=\sim_{f^k}$ for any $k$. Since $f$ is
a train track for $\sim_f$, in particular $\sim_f$ is a train track structure, so any vertex
has at least two gates. Conditions $1)$ and $2)$ of Definition~\ref{Defttm}, that hold for $f$, imply that $f^k(e)$ is a legal path, hence Condition $1)$
for $f^k$.

If $\tau$ is a turn and $f^k(\tau)$ is $\sim_{f^k}$-illegal, then $f^k(\tau)$ is $\sim_f$-illegal,
which implies that $\tau$ is $\sim_f$-illegal, and so $\sim_{f^k}$-illegal.\qed

\begin{Lem}~\label{lem:lemmaZ}
  If $TT_0(\Phi)\neq\emptyset$, then $\lambda_{\Phi^k}=(\lambda_{\Phi})^k$.
\end{Lem}
\proof Let $T\in TT_0(\Phi)\subseteq \mathcal M(\Phi)$ and let $f:T\to T$ be an optimal train
track map with respect to $\sim_f$.
By Lemma~\ref{lem:lemmaY} $f^k$ is a train track map for $\sim_{f^k}$, in particular
$\Lip(f^k)=\Lip(f)^k$,
$\sim_{f^k}$ is a train track structure and $f^k$ is optimal by Remark~\ref{ropt}. By
Theorem~\ref{thm_optimal_map} $(\lambda_{\Phi})^k=\Lip(f^k)=\lambda_{\Phi^k}$.\qed

We can now conclude the proof of Theorem~\ref{thm:A=M}. Let $T\in \mathcal M(\Phi)$, and let
$f:T\to T$ be an optimal map.
By
Lemmas~\ref{lem:lemmaZ} and~\ref{lem:extract} all the iterates $f^k$ are optimal.

We claim that
$f$ is a train track map with respect to $\simfk$. First, note that since $f^k$ is optimal, every vertex has at
least two gates, hence $\simfk$ is a train track structure.

Now, we check Condition $1)$ of Definition~\ref{Defttm}. Suppose that an edge $e$ is folded by some
$f^k$, and choose the first $k$ so that this happens. Let $p$ be a point interior to $e$ where
a fold occurs.

By
Theorem~\ref{thm:A_0=M} there is $T_t$ obtained from $T$ by a finite number of as-small-as-required folds so that
$T_t\in TT_0(\Phi)$. Without loss of generality we may suppose that $T_t$ is obtained by $T$
by a simple fold, and show that in this case $T\in TT(\Phi)$.
 Let $h_t:T_t\to T_t$ be the map induced by the fold as in Figure~\ref{fig:ht}. We choose the
 fold small enough so that $p$ remains in the interior of the edge $e_t$ corresponding to $e$.
The following diagram commutes
$$
\xymatrix{A\ar[r]^f\ar[d]_{q_t} & \Phi(A)\ar[r]^{f^2}\ar[d]_{q_t} & \Phi(A)
\ar[d]_{q_t}\ar[r]^{f^3} &
\Phi(A)\ar[d]_{q_t} &\dots\ar[r]^{f^k}&\Phi(A)\ar[d]_{q_t}
\\
A_t\ar[ur]_{f_t}\ar[r]_{h_t} & \Phi(A_t)\ar[ur]_{f_t}\ar[r]_{(h_t)^2} & \Phi(A_t)
\ar[ur]_{f_t}\ar[r]_{(h_t)^3} & \Phi(A_t)&\dots\ar[r]_{(h_t)^k} &\Phi(T_t)
}
$$
and therefore $e_t$ is folded at $p$ by $(h_t^k)$. But in the proof of Theorem~\ref{thm:A_0=M}
we have seen that the $h_t$ are train track maps, so edges are never folded.

As for Condition $2)$ of Definition~\ref{Defttm}, note that
Condition $1)$ and the definition of $\simfk$ imply that
condition $2)$ is automatically satisfied.\qed

\medskip

\begin{Rem}\label{r22}
We notice that in~\cite{HM} the authors ask if, given $\Phi$, there is $N$ such that $\mathcal
A_\Phi$ is the closure of $\cup_{i=1}^N TT(\Phi^i)$. In the same work they provide an example
of a $\Phi$ and point in $X\in \mathcal A_\Phi$ not supporting any train track for any
$\Phi^i$. 
This example, together with our Theorem~\ref{thm:A=M}, provides a negative answer to the
question,  since $TT(\Phi^i)=\mathcal M(\Phi)$ is closed, and hence any finite union of $TT(\Phi^i)$ is closed. 
\end{Rem}

\begin{Thm}[Folding axis]
 Let $\Phi$ be an $\mathcal O$-irreducible
  element of $\Aut(G,\mathcal O)$. Then $TT_0(\Phi)$ is
invariant under folding directed by optimal train track maps.

More precisely,  if $f:T\to \Phi(T)$ is train track map with respect to $\sim_f$,  and $T_t$ is an isometric
    folding path from $T\to \Phi(T)$ directed by $f$, then  the induced map
   $h_t:T_t\to\Phi(T_t)$ is a train
    track map with respect to $\sim_{h_t}$.
\end{Thm}
\proof Since $f$ is an optimal train track map,
the folding path from $T$ to $\Phi^2(T)$ directed by $f^2$ is the concatenation of the
folding path $\gamma_0$ from $T$ to $\Phi(T)$ directed by $f$ and $\Phi(\gamma_0)$.
Therefore we can form a folding line directed by $f$ concatenating the paths $\Phi^k(\gamma_0)$.

Let $g$ be  an element such that $\axis_T(g)$ is legal and $f^k(\axis_T(g))$ is legal for any
$k$. (Such an element exists because $f$ is a train track map.) It follows that $\axis_T(g)$ is
never folded during the folding procedure, so $\axis_{T_t}(g)$ is legal and
$h_t^k(\axis_{T_t}(g))$ is legal. Thus $h_t$ is a train track map as desired.\qed

It is useful to have train track maps that respect the simplicial structure
(i.e. that map vertices to vertices). The presence of non-free vertices is in this case an
advantage with respect to the classical case (\cite{BestvinaBers,BestvinaHandel}). We give a
detailed proof of the following result in full generality by completeness. We notice that we
make no use of Perron-Frobenius theory.

\begin{Thm}[Simplicial train track]
  Let $\Phi\in\Aut(G,\mathcal O)$ be irreducible. Then there exists a simplicial optimal train track
  map representing $\Phi$. More precisely, if $T\in TT(\Phi)$, then the closed simplex of $T$
  contains a point admitting a simplicial (optimal) train track map.
\end{Thm}
\proof
The idea is to ``snap'' images of vertices to nearest vertices, as suggested
in~\cite{BestvinaBers}. Let $T\in TT(\Phi)$ and $f:T\to T$ be an optimal train track map (with
respect to $\langle\sim_{f^k}\rangle$) representing $\Phi$. Let $\lambda=\lambda_\Phi>1$ be the
Lipschitz constant of $f$.

We will argue by induction on the number of orbits of vertices whose image is not a vertex,
(note in particular that such vertices are free,)
making use of local surgeries for the inductive step.

\medskip

First, we describe in details the local move that we use, and after we will adjust the map
$f$.
The moves can be interpreted as local isometric folds followed by local isometric ``unfolds''.
However, we describe them in terms
of surgeries because this viewpoint helps in controlling the derivative of $f$.
We remark that we are not working with covolume-one trees, thus no rescaling is needed.

\medskip

Choose $\varepsilon>0$ small enough so that:
\begin{enumerate}
\item $B(w,\varepsilon)\cap B(w',\varepsilon)=\emptyset\ \forall w,w'\in VT: w\neq w'$.
\item $f(B(w,\varepsilon))\cap B(w',\varepsilon)=\emptyset \ \forall w,w'\in VT: f(w)\neq w'$.
\end{enumerate}

Let $v\in VT$ be such that there is $k\geq 0: f^k(v)\notin VT$. In particular $v$ is free and
has two gates, that we label as positive and negative. We build an isometric model of
$B(v,\varepsilon)$ as follows. By our choice of $\varepsilon$, $B(\varepsilon,v)$ is
star-shaped with say $n_-$ negative and $n_+$ positive strands.
Therefore, $B(\varepsilon,v)$ is isometric to the space obtained from $n_-$ copies of $(-\varepsilon,0]$ and
$n_+$ copies of $[0,\varepsilon)$ by gluing the $0$'s. See Figure~\ref{fig:snap}, left side.
\setlength{\unitlength}{.9ex}
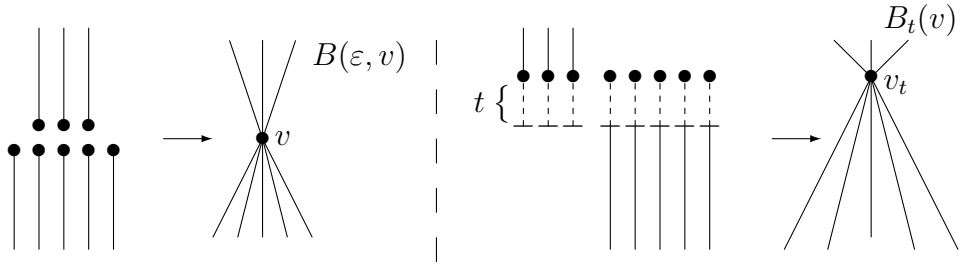
\begin{figure}[h]
  \centering
  \begin{picture}(74,18)
    \multiput(0,0)(2,0){5}{\line(0,1){8}}
    \multiput(0,8)(2,0){5}{\makebox(0,0){$\bullet$}}
    \multiput(2,10)(2,0){3}{\line(0,1){8}}
    \multiput(2,10)(2,0){3}{\makebox(0,0){$\bullet$}}

\put(12,9){\vector(1,0){4}}

\put(20,1){\line(0,1){16}}
\put(18,1){\line(1,4){2}}
\put(16,1){\line(1,2){4}}
\put(22,1){\line(-1,4){2}}
\put(24,1){\line(-1,2){4}}

\put(20,9){\line(-1,3){2.66}}
\put(20,9){\line(1,3){2.66}}
\put(20,9){\makebox(0,0){$\bullet$}}

\put(24,15){$B(\varepsilon,v)$}
\put(21,8.5){$v$}

\multiput(34,-1)(0,4){5}{\line(0,1){2}}

\put(39,0){
    \multiput(9,0)(2,0){5}{\line(0,1){10}}
    \multiput(9,10)(2,0){5}{\multiput(0,0)(0,1){4}{\line(0,1){.5}}}
    \multiput(9,14)(2,0){5}{\makebox(0,0){$\bullet$}}
    \multiput(9,10)(2,0){5}{\makebox(0,0){$-$}}
    \multiput(2,14)(2,0){3}{\line(0,1){4}}
    \multiput(2,10)(2,0){3}{\multiput(0,0)(0,1){4}{\line(0,1){.5}}}
    \multiput(2,14)(2,0){3}{\makebox(0,0){$\bullet$}}
    \multiput(2,10)(2,0){3}{\makebox(0,0){$-$}}
 \put(-.5,11.2){$\big\{$}
 \put(-2,11.2){$t$}

 \put(10,0){
   \put(12,9){\vector(1,0){4}}

   \put(20,1){\line(0,1){16}}

   \put(20,14){\line(1,-2){7}}
   \put(20,14){\line(1,-4){3.5}}
   \put(20,14){\line(-1,-2){7}}
   \put(20,14){\line(-1,-4){3.5}}

   \put(20,14){\line(-1,1){3}}
   \put(20,14){\line(1,1){3}}

   \put(20,14){\makebox(0,0){$\bullet$}}

   \put(21,18){$B_t(v)$}
   \put(21,13){$v_t$}
 }
}
  \end{picture}
  \caption{Local models for $B(\varepsilon,v)$ and $B_t(v)$.}
  \label{fig:snap}
\end{figure}

For $|t|<\varepsilon$, let $B_t(v)$ be the space obtained from $n_-$ copies of $(-\varepsilon,t]$ and
$n_+$ copies of $[t,\varepsilon)$ by gluing the points $t$'s. See Figure~\ref{fig:snap}, right side.

$B_t(v)$ has a vertex, corresponding to the endpoints ``$t$'s'', which we denote by $v_t$, and the
boundary of $B_t(v)$  is naturally identified with that of $B(v,\varepsilon)$.

\medskip

Now, we cut out from $T$ the whole $G$-orbit of $B(v,\varepsilon)$
and we paste back copies of $B_t(v)$
 using the natural identifications on the boundaries. We say that {\bf we equivariantly moved $v$ by $t$}.

\medskip
Now, choose $v\in VT$ with $f(v)\notin VT$ and move it by $t<\varepsilon$ in the direction
given by the nearest vertex to $f(v)$. (If $f(v)$ is a midpoint of an edge we chose a direction.)

Define $pre(v)=\{w\in VT\ :\ f^k(w)=v$ for some $k\geq 0$ and $f^i(w)\in VT$ for
all $0\leq i\leq k\}$. Thus, $pre(v)$ is finite and consists of the iterate $f$-pre-images of $v$.
Any $w\in pre(v)$ is free and if $w\neq v$ then $w\notin pre(w)$.
Note also that $Gv\cap pre(v)=v$.  Moreover, since $f$ is a train track
map, any $w\in pre(v)$ has two gates, with positive and negative labels determined by that at
$v$ via $f^k$.

For $w\in pre(v)$, if $k$ is the first power so that $f^k(w)=v$, we consider the ball $B(w,\varepsilon/\lambda^{k})$ and we
equivariantly move $w$ by
$t/\lambda^{k}$. Note that this is possible because such balls are all disjoint from the $G$-orbit of each other, and disjoint from
$f(B(v,\varepsilon))$ because  ($\lambda>1$ and) our choice of $\varepsilon$.

\medskip

We denote by $T_t$ the tree obtained from $T$ in such a way.

\medskip

We are now left to define $f_t:T_t\to T_t$. Let $N$ be the union of the $G$-orbits of
of the metric balls $B(w,\varepsilon/\lambda^k)$ for all $w\in pre(v)$ and $k$ as above, and let
$N_t$ be the union of the corresponding sets $B_t's$ (see Figure~\ref{fig:snap}). Thus $T_t=(T\setminus N)\cup N_t$.

\medskip

On the set $f^{-1}(T\setminus N)\cap(T\setminus N)$ we set $f_t=f$. Clearly, there $\dot
f_t=\lambda$.

\medskip
Let $\sigma\subset f^{-1}(N)\setminus N$ be a segment without vertex in its interior. As
$\sigma$ is connected, $f(\sigma)$ is contained in one of the balls $B(w,\varepsilon)$.
Since $f$ is a train track map, edges are mapped to legal paths. Therefore, $f(\sigma)$ is
contained in the union of a negative and a positive strand of $B(w,\varepsilon)$. Such union is
isometric to the union of the corresponding strands in $B_t(w)$. We define $f_t$ on $\sigma$ by
composing $f$ with such isometry. Clearly $\dot f_t=\lambda$.

\medskip
It remains to define $f_t$ on $N_t$. For any $w\in pre(v)$, any strand $S$ of $B_t(w)$ corresponds
isometrically to a legal path $\sigma$ in $B(w,\varepsilon/\lambda^k)$, which is uniformly
stretched by $f$
by factor $\lambda$. Since we moved $w$ by $t/\lambda^{k}$ and $f(w)$ by $t/\lambda^{k-1}$,
$f(\sigma)\subset B(f(w),\varepsilon/\lambda^{k-1})$ corresponds isometrically to a strand $S_1$
in $B_t(f(w))$ (or $f(B(v,\varepsilon))$ if $w=v$). Therefore $f_t$ is defined by pre- and
post-composing $f$ with such isometries. Clearly $\dot f_t=\lambda$. See Figure~\ref{fig:snap3}.

\setlength{\unitlength}{.9ex}
\begin{figure}[h]
  \centering
  \begin{picture}(82,20)

   \put(-12,0){
     \put(20,14){\line(0,1){3}}
     {\linethickness{1.75pt}
     \put(20,3){\line(0,1){11}}}
     \put(20,14){\line(1,-2){7}}
     \put(20,14){\line(1,-4){3.5}}
     \put(20,14){\line(-1,-2){7}}
     \put(20,14){\line(-1,-4){3.5}}

     \put(20,14){\line(-1,1){3}}
     \put(20,14){\line(1,1){3}}

     \put(17,19){$B_t(w)$}
     \put(19,0){$S$}
   }

   \put(13,10){\vector(1,0){10}}
   \put(15,11){Isom}
   \put(7,0){
     \put(20,14){\line(0,1){3}}
     {\linethickness{1.75pt}\put(20,3){\line(0,1){11}}}
     \put(18,1){\line(1,4){2}}
     \put(16,1){\line(1,2){4}}
     \put(22,1){\line(-1,4){2}}
     \put(24,1){\line(-1,2){4}}

     \put(20,9){\line(-1,3){2.66}}
     \put(20,9){\line(1,3){2.66}}
     \put(20,9){\makebox(0,0){$\bullet$}}

     \put(14,19){$B(\frac{\varepsilon}{\lambda^k},w)$}
     \put(17,8.5){$w$}
   }

   \put(32,11){\vector(4,1){10}}
   \put(32,9){\vector(4,-1){10}}
   \put(37,9.5){$f$}
   \put(26,0){
     \put(20,14){\line(0,1){3}}
     {\linethickness{1.75pt}\put(20,1){\line(0,1){15}}}
     \put(18,1){\line(1,4){2}}
     \put(22,1){\line(-1,4){2}}

     \put(20,9){\makebox(0,0){$\bullet$}}

     \put(12,19){$B(\frac{\varepsilon}{\lambda^{k-1}},f(w))$}
     \put(21,7){$f(w)$}
   }

   \put(53,10){\vector(1,0){10}}
   \put(55,11){Isom}

   \put(48,0){
     \put(20,1){\line(0,1){18}}
     {\linethickness{1.75pt}\put(20,1){\line(0,1){15}}}
     \put(20,16.5){\line(1,-4){3.5}}
     \put(20,16.5){\line(-1,-4){3.5}}
     \put(15,19){$B_t(f(w))$}
     \put(19,-2){$S_1$}
   }

  \end{picture}
  \caption{The definition of $f_t$ on $B_t(w)$.}
  \label{fig:snap3}
\end{figure}
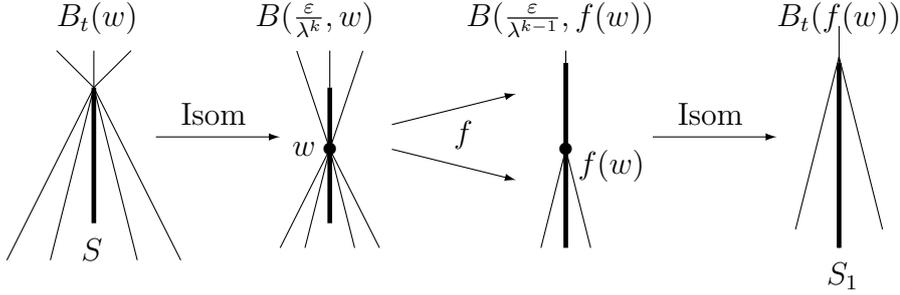

It is clear that $(T,f)$ is homeomorphic to $(T_t,f_t)$, therefore $f_t$ still is an optimal
train track representing $\Phi$.

We remark that when we move $v$ by $t$, then $f(v)$ moved toward its nearest vertex $u$ by
$\lambda t$. On the other hand, even if $u$ has been moved, that was by an amount of
$t/\lambda^{k}$ for some
$k\geq 0$. Therefore $f(v)$ approaches $u$ at speed at least $\lambda -1>0$.

We can finally run the induction on the number of orbits of vertices whose image is not a
vertex. Let $v$ be such a vertex.
We move $v$ as described as long as we can.
Since $d(f(v),VT)$ is strictly decreasing, the process must stop. The process stops when we
cannot chose $\varepsilon>0$ with the required properties. That is to say, when either
$f(v)$ is a vertex or some moved vertex collided with another vertex $v'$.
In both cases we decreased by one our induction parameter.
\qed

The following is a direct corollary of the existence of train track maps for free products. It
was proved in~\cite{BestvinaHandel} for free groups and in~\cite{CollinsTurner} for free products.

\begin{Cor}[\cite{BestvinaHandel},\cite{CollinsTurner}]
Let $G$ be a group acting co-compactly on a tree with trivial edge stabilisers and freely indecomposable vertex stabilisers.
Then any $\Phi \in Out(G)$ has a representative which is a relative train track map. In particular, relative train tracks for free groups exist.
\end{Cor}
\proof We can write $G=G_1* \ldots *G_p*F_k$, where the $G_i$ are freely indecomposable and non-free.
By the Kurosh subgroup theorem, any subgroup of $G$ can be written as a free product
of conjugates of subgroups of the $G_i$ and some free group. So if $H \leq G$, then $H \cong A_1* \ldots A_m * F_l$, for some $A_i \neq 1$ which are conjugates
of subgroups of the $G_i$ and some free group $F_l$ of rank $l$. Define the Kurosh rank of such a subgroup $H$ to be $m+l$, denoted $\kappa(H)$.
Note this number may be infinite in general, but will certainly be finite if $H$ is a free factor (and in many other cases).

Now define the reduced Kurosh rank of $H$ to be $\overline{\kappa}(H)=\max(0, \kappa(H)-1)$.
The Kurosh rank of a free factor system, $\mathcal{G}= \{ [G_i] \}$  is then defined to be $\kappa({\mathcal G})= \sum {\kappa(G_i) }$ and the reduced Kurosh rank of $\mathcal{G}$ is defined
to be $\overline{\kappa}({\mathcal G}) = \sum \overline{\kappa}(G_i) $.

These are finite numbers, and if $\mathcal{G} \sqsubseteq \mathcal{H}$ then $\overline{\kappa}({\mathcal G})  \leq \overline{\kappa}({\mathcal H})$ and $\kappa({\mathcal G})  \leq \kappa({\mathcal H})$.
Moreover, if $\mathcal{G} \sqsubset \mathcal{H}$, then either $\overline{\kappa}({\mathcal G})  < \overline{\kappa}({\mathcal G})$ or $\kappa({\mathcal G})  <  \kappa({\mathcal H})$.

Hence given $\Phi \in Out(G)$ there is a maximal $\Phi$-invariant, proper free factor system, with corresponding space of trees $\mathcal{O}$.
A simplicial train track map representing $\Phi$ for some tree in $\mathcal{O}$ is a relative train track map in the sense of~\cite{CollinsTurner}. \qed


\section{Computing stretching factors}
This section is devoted to prove that stretching factors are realized by a class of particularly
simple elements. We generalize the line used in~\cite{FrancavigliaMartino}, taking in account
possible pathologies coming from the presence of non-free vertices. We remark that even if this
section is at the end of the paper, the results of this section
are independent from those in Section~\ref{sectionTT} (where
Theorem~\ref{sausages} is used).

\begin{Not}
Let $x,y,z,t$ be vertices of a $G$-tree, not necessarily different from each other. We write
$$\nopod{y}{x}{z}\Leftrightarrow x\in[y,z] \qquad\text{ and }\qquad
\tripod{y}{x}{z}\Leftrightarrow x\notin[y,z].$$
Note that \tripod{$y$}{$x$}{$z$}, if and only if the segments $[y,x]$ and $[x,z]$ intersect in a
sub-segment starting at $x$ which is not a single point. We will use the following two
inference rules, whose verification is immediate.
$$\tripod{$y$}{$x$}{$z$} + \tripod{$z$}{$x$}{$t$} = \tripod{$y$}{$x$}{$t$}$$
$$\tripod{$y$}{$x$}{$z$} + \nopod{$z$}{$x$}{$t$} = \nopod{$y$}{$x$}{$t$}$$

\end{Not}

We will now be concerned in finding {\em good} elements $g$ with a tight axis. In the
subsequent discussion when we say ``$g$ does not have a tight axis'' we mean either that the
axis of $g$ is not tight or that $g$ is elliptic.

\begin{Lem}\label{l:notight}
Let $A,B\in \mathcal O$ and $f:A\to B$ be a PL-map. Let $L\subset A$ be a tight sub-tree
isomorphic to $\mathbb R$. Suppose there is $g\in G$ and $x\in L$ so that $x\neq gx$ and
$[x,gx]\subset L$. Set $y=f(x)$. If $g$ does not have tight axis, then we have $$\tripod{$g^{-1}y$}{$y$}{$gy$}.$$
\end{Lem}
\begin{proof}
Suppose $g$ has a fixed point $v$. Since $A$ is a tree, $v$ is the middle point of $[x,gx]$,
and in particular we have \tripod{$g^{-1}x$}{$x$}{$gx$} (where the center of the tripod is
exactly $v$). The claim follows.

If $g$ has no fixed point, then it has an axis. The concatenation of the segments
$[g^kx,g^{k+1}x]$, as $k$ varies in $\mathbb Z$, is a $g$-invariant tree. Therefore it contains the axis of $g$.
Thus, $\axis_A(g)$ is contained in the $g$-orbit of
$[x,gx]$. In particular $\axis_A(g)\subset A_{\max}$ because $[x,gx]\subset L \subset
A_{\max}$ which is a $G$-invariant sub-set of $A$.
By hypothesis  the axis of $g$ is not tight, thus $f|_{\axis_A(g)}$ is not injective. Since $L$
is tight, then $f|_L$ is injective. Therefore $f$ must overlap an initial segment of $[x,gx]$
with a terminal segment of $[g^{-1}x,x]$ which is exactly the claimed formula.
\end{proof}

\begin{Lem}[No triple points]~\label{no_triple_points}
Let $A,B\in\mathcal O$, $f\colon A\to B$ be a PL-map, and $g\in G$ be such that
$\axis_A(g)$ is tight.
If there exists $x\in \axis_A(g)$ such that $|Gx\cap [x,gx]|\geq 4$ then, there exists $h\in G$
with tight axis, such that $l_A(h)<l_A(g)$.
\end{Lem}
\begin{proof}
By hypothesis, there exists $x\in A$ and $a,b\in G$ such that $\axis_A(g)$ locally looks as
depicted in Figure~\ref{fig:vvgsegment}.
\begin{figure}[h]
\centerline{
\xymatrix{
\dots \stackrel{{x}}{\bullet}\ar@{-}[r] &\stackrel{ax}{\bullet} \ar@{-}[r] & \stackrel{bx}{\bullet}\ar@{-}[r] & \stackrel{gx}{\bullet} \dots
}}
\caption{The segment $[x,gx]$}\label{fig:vvgsegment}
\end{figure}

In other words, the segment $[x,gx]$ is the concatenation of segments $[x,ax],$ $[ax,bx]$ and
$[bx,gx]$. Let $y=f(x)$. Since $\axis_A(g)$ is tight, we have a similar situation in $B$. See~Figure~\ref{fig:uugsegment}.
\begin{figure}[h]
\centerline{
\xymatrix{
\dots \stackrel{{y}}{\bullet}\ar@{-}[r] &\stackrel{ay}{\bullet} \ar@{-}[r] & \stackrel{by}{\bullet}\ar@{-}[r] & \stackrel{gy}{\bullet} \dots
}}\caption{The segment $[y,gy]$}\label{fig:uugsegment}
\end{figure}
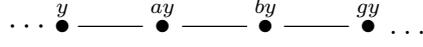
In particular we have
$$\delta:=\nopod{$y$}{$ay$}{$by$}\quad \qquad
b^{-1}\delta:=\nopod{$b^{-1}y$}{$b^{-1}ay$}{$y$}
\qquad \text{ and }\quad
a^{-1}\delta:=\nopod{$a^{-1}y$}{$y$}{$a^{-1}by$}
$$

We look at $a, ba^{-1}, gb^{-1}$ (corresponding to single steps in Figure~\ref{fig:uugsegment}).
Clearly, all of them have $A$-length strictly smaller than that of $g$. Thus, if one of them has tight
axis we are done. We can therefore suppose that none of them has tight axis. In particular, by
Lemma~\ref{l:notight} we know
$$\alpha:=\tripod{$a^{-1}y$}{$y$}{$ay$}
\qquad
\beta:=\tripod{$ab^{-1}ay$}{$ay$}{$by$}
.$$
Now we look at $b$ (corresponding to a double step in Figure~\ref{fig:uugsegment}). As above we
have $l_A(b)<l_A(g)$. We argue by contradiction assuming that $b$ has no tight axis.
Then by Lemma~\ref{l:notight} we have \tripod{$by$}{$y$}{$b^{-1}y$}. Moreover, by assumption we
have $ay\in(y,by)$ whence $b^{-1}ay\in(b^{-1}y,y)$; from which we get

$$\chi:=\tripod{$ay$}{$y$}{$b^{-1}ay$}.$$

 It follows that
$$\alpha+\chi+(a^{-1}\beta)=
\tripod{$a^{-1}y$}{$y$}{$ay$}
+\tripod{$ay$}{$y$}{$b^{-1}ay$}
+\tripod{$b^{-1}ay$}{$y$}{$a^{-1}by$}
=
\tripod{$a^{-1}y$}{$y$}{$a^{-1}by$}$$
 which contradicts $a^{-1}\delta=\nopod{$a^{-1}y$}{$y$}{$a^{-1}by$}$.

\end{proof}

\begin{Def}
When the hypothesis of Lemma~\ref{no_triple_points} are satisfied we say that $\axis_A(g)$ has
triple points.
\end{Def}

\begin{Lem}[Four points lemma]\label{l:4pt}
Let $A,B\in\mathcal O$ and $f\colon A\to B$ be a PL-map. Let $L\subset A$ be a tight sub-tree
isomorphic to $\mathbb R$. Suppose there is $x\neq v\in L$ and $a,b\in G$ such that
$ax\neq bv$ and $ax,v\in[x,bv]\subset L$ (see Figure~\ref{fig:xtab}). Let $y=f(x)$, $w=f(v)$. If \tripod{$a^{-1}y$}{$y$}{$ay$} and
\tripod{$b^{-1}w$}{$w$}{$bw$},
then $b^{-1}a$ has a tight axis, given by the iterates of $[x,b^{-1}ax]$.
\end{Lem}
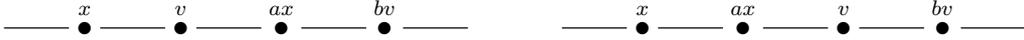
\begin{figure}[h]
\centerline{
\xymatrix{
\ar@{-}[r]&  \stackrel{x}{\bullet}\ar@{-}[r] &\stackrel{v}{\bullet} \ar@{-}[r] &
\stackrel{ax}{\bullet}\ar@{-}[r] & \stackrel{bv}{\bullet} \ar@{-}[r] &
}
\qquad
\xymatrix{
\ar@{-}[r]& \stackrel{{x}}{\bullet}\ar@{-}[r] &\stackrel{ax}{\bullet} \ar@{-}[r] &
\stackrel{v}{\bullet}\ar@{-}[r] & \stackrel{bv}{\bullet} \ar@{-}[r] &
}
}\caption{The two possibilities for the segment $[x,bv]$}\label{fig:xtab}
\end{figure}
\begin{proof}
First of all, note that the tripodal hypotheses imply $x\neq ax$ and $v\neq bv$ (and thus $x\neq bv$).
We have to show that the tree formed by the iterates of the segment
$[x,b^{-1}ax]$
is a  tight line. Clearly it is contained in
$A_{\max}$ because $L$ is, and $[x,b^{-1}ax]$ is contained in the union of the two
segments $[x,v]$ and $[v,b^{-1}ax]$ (the latter may be a single point). Note that a priori we may have
\tripod{$x$}{$v$}{$b^{-1}ax$}. Since $[v,b^{-1}ax]=b^{-1}([bv,ax])$ and $[ax,bv]\subset L$,
we know that $f$ is injective on both $[x,v]$ and $[v,b^{-1}ax].$ Therefore,  we are left
 to prove  $$I:=\nopod{$y$}{$w$}{$b^{-1}ay$} \qquad \textrm{ and } \qquad II:= \nopod{$a^{-1}by$}{$y$}{$b^{-1}ay$}.$$

First we prove $I$. Since both $ax$ and $v$ lie in $[x,bv]$, $ax\neq bv$ gives
 $\tripod{$ax$}{$bv$}{$v$}.$ Applying $f$ we get
\tripod{$ay$}{$bw$}{$w$} whence, acting with $b^{-1}$,
\tripod{$b^{-1}ay$}{$w$}{$b^{-1}w$}. By hypothesis we know \tripod{$b^{-1}w$}{$w$}{$bw$}. Summing
up we get
$$\tripod{$b^{-1}ay$}{$w$}{$b^{-1}w$}+ \tripod{$b^{-1}w$}{$w$}{$bw$}=
\tripod{$b^{-1}ay$}{$w$}{$bw$}.$$ Now, if we had \tripod{$y$}{$w$}{$b^{-1}ay$} we would get
$\tripod{$y$}{$w$}{$b^{-1}ay$}+\tripod{$b^{-1}ay$}{$w$}{$bw$}=\tripod{$y$}{$w$}{$bw$}$
which is impossible because $L$ is tight. Thus \tripod{$y$}{$w$}{$b^{-1}ay$} is not true and
$I$ is true. Note that this implies that $a\neq b$.

We now prove $II$. From tightness of $L$ we get \nopod{$bw$}{$ay$}{$y$}
whence \nopod{$a^{-1}bw$}{$y$}{$a^{-1}y$} and thus
\nopod{$a^{-1}bw$}{$y$}{$a^{-1}y$}+\tripod{$a^{-1}y$}{$y$}{$ay$}=\nopod{$a^{-1}bw$}{$y$}{$ay$}.

Acting on $I$ by $a^{-1}b$ we get
\nopod{$a^{-1}by$}{$a^{-1}bw$}{$y$} that, together with \nopod{$a^{-1}bw$}{$y$}{$ay$}, implies
\nopod{$a^{-1}by$}{$y$}{$ay$} because by hypothesis $ax\neq bv$ and so $y\neq a^{-1}bw$. As above, since both $ax,v$ lie in $[x,bv]$
we have  \tripod{$ax$}{$x$}{$v$} whence
\tripod{$ay$}{$y$}{$w$}. It follows \nopod{$a^{-1}by$}{$y$}{$w$} that, together with
\nopod{$y$}{$w$}{$b^{-1}ay$} gives
$$\nopod{$a^{-1}by$}{$y$}{$b^{-1}ay$}=II.$$

\end{proof}

 \begin{Lem}[No crossing points]\label{no_crossing_points}
Let $A,B\in\mathcal O$, $f\colon A\to B$ be a PL-map, and $g\in G$ be such that
$\axis_A(g)$ is tight.
Suppose that there exists points $x,v\in \axis_A(g)$ and $a,b\in G$ such that, with respect to the linear order
of $\axis_A(g)$, we have $x<v<ax<bv<gx$ (see Figure~\ref{fig:crossings})
\begin{figure}[h]
\centerline{
\xymatrix{
\dots \stackrel{{x}}{\bullet}\ar@{-}[r] &\stackrel{v}{\bullet} \ar@{-}[r] & \stackrel{ax}{\bullet}\ar@{-}[r] & \stackrel{bv}{\bullet} \ar@{-}[r] & \stackrel{gx}{\bullet}\dots
}}
 \caption{The segment $[x,gx]_A$}\label{fig:crossings}
\end{figure}
Then, there exists $h\in G$
with tight axis, such that $l_A(h)<l_A(g)$.
 \end{Lem}

 \begin{proof} Note that we may have $a=b$.
Let $y=f(x)$ and $w=f(v)$. Since $[x,ax],$ $[v,bv]$ lie in $\axis_A(g)$, they lie in
$A_{\max}$. Also $l_A(a),l_A(b)< l_A(g)$. If one of them has tight axis, then the claim is
proved by letting $h=a$ or $h=b$. Otherwise, by Lemma~\ref{l:notight} we have
\tripod{$a^{-1}y$}{$y$}{$ay$} and \tripod{$b^{-1}w$}{$w$}{$bw$} and by Lemma~\ref{l:4pt}
the element $h=b^{-1}a$ has a tight axis (in particular, in this case $a\neq b$). On the other hand $l_A([hx,x])=l_A([ax,bx])\leq
l_A([ax,bv])+l_A([x,v])<l_A(g)$, therefore $l_A(h)<l_A(g)$.
\end{proof}

\begin{Def}
When the hypothesis of Lemma~\ref{no_crossing_points} are satisfied we say that $\axis_A(g)$
has crossing points.
\end{Def}

\begin{Lem}[No bad triangles]\label{no_bad_triangles}
Let $A,B\in\mathcal O$, $f\colon A\to B$ be a PL-map, and $g\in G$ be such that
$\axis_A(g)$ is tight. If there exists $x,v,t\in \axis_A(g)$ and $a,b,c\in G$ such that, with
respect to the linear order of $[x,gx]$, we have $x<ax<v<bv<t<ct<gx$ (See Figure~\ref{fig:abgsegment})
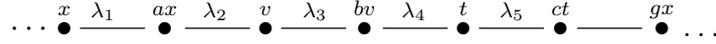
\begin{figure}[h]
\centerline{
\xymatrix{
\dots \stackrel{{x}}{\bullet}\ar@{-}[r]^{\lambda_1} &
\stackrel{ax}{\bullet} \ar@{-}[r]^{\lambda_2} &
\stackrel{v}{\bullet}\ar@{-}[r]^{\lambda_3} &
\stackrel{bv}{\bullet} \ar@{-}[r]^{\lambda_4} &
\stackrel{t}{\bullet}\ar@{-}[r]^{\lambda_5} &
\stackrel{ct}{\bullet} \ar@{-}[r] &
\stackrel{gx}{\bullet}\dots
}
}
\caption{The local situation in $\axis_A(g)$}\label{fig:abgsegment}
\end{figure}
 then there exists $h\in G$ with tight axis, such that
$l_A(h)<l_A(g)$.
\end{Lem}
\begin{proof}
Let $y=f(x), w=f(v)$ and $s=f(t)$. We first try $h=a,b,c$. Since $[x,ax],$ $[v,bv]$ and
$[t,ct]$ lie in $A_{\max}$ and $l_A(x),l_A(y),l_A(z)<l_A(g)$, if one of them has a tight axis we
are done. Otherwise, by Lemma~\ref{l:notight} we have
\tripod{$a^{-1}y$}{$y$}{$ay$},  \tripod{$b^{-1}w$}{$w$}{$bw$} and
  \tripod{$c^{-1}s$}{$s$}{$cs$}.  From Lemma~\ref{l:4pt} we deduce that $b^{-1}a$, $c^{-1}b$ and
  $c^{-1}a$ all have tight axis. It suffices to show that one of them has length less than
  $g$ in $A$. Let $\lambda_1=l_A([x,ax]), \lambda_2=l_A([ax,v]), \lambda_3=l_A([v,bv]),
  \lambda_4=l_A([bv,t]), \lambda_5=l_A([t,ct])$.

$$l_A(b^{-1}a)\leq l_A([x,b^{-1}ax])\leq l_A([x,v])+l_A([v,b^{-1}ax])=\lambda_1+\lambda_2+\lambda_2+\lambda_3$$
$$l_A(c^{-1}b)\leq l_A([bv,ct])+l_A([v,t])=\lambda_4+\lambda_5+\lambda_3+\lambda_4$$
Summing up
$$l_A(b^{-1}a)+l_A(c^{-1}b)<2(\lambda_1+\lambda_2+\lambda_3+\lambda_4+\lambda_5)\leq 2l_A(g)$$
So one of them has length strictly less than $g$ in $A$.
\end{proof}

\begin{Def}
When the hypothesis of Lemma~\ref{no_bad_triangles} are satisfied we say that $\axis_A(g)$
has bad triangles.
\end{Def}

\begin{Thm}[Sausage Lemma and candidates]
\label{sausages}
  Let $A,B\in\mathcal O$. Then the minimal stretching factor $\Lambda_R(A,B)$ is realized by an
  element $g$ such that the projection of $\axis_A(g)$ to $\quot GA$ is of the form:
  \begin{enumerate}[i)]
  \item Embedded simple loop: $O$;
  \item embedded figure-eight: $\infty$ (a bouquet of two copies of $S^1$);
  \item embedded barbel: $O-O$ (two simple loops joined by a segment);
  \item embedded singly degenerate barbell: $\bullet-O$ (a non-free vertex and a simple loop joined by
    a segment);
  \item embedded doubly degenerate barbell: $\bullet-\bullet$ (two non-free vertices joined by
    a segment).
  \end{enumerate}
The loops and segments above may contain free and non-free vertices.
\end{Thm}
\begin{proof}
By Corollary~\ref{cor:optimalfromf} and Theorem~\ref{thm_optimal_map} we know that there is an optimal map $f:A\to B$ and an element
$g$ so that $\axis_A(g)$ is tight. Translation lengths of hyperbolic elements form a discrete
set, so we may assume that $g$ has minimal translation length among those elements with tight
axis. In order to prove our claim it suffices to find an element
with tight axis whose projection to $\quot GA$ is of one of the types $i),\dots,v)$.
By Lemmas~\ref{no_triple_points},~\ref{no_crossing_points}
and~\ref{no_bad_triangles} we know that $\axis_A(g)$ has no triple points,
nor crossing points, nor bad triangles.

Choose $x_0\in \axis_A(g)$. Since there are no triple points in $\axis_A(g)$ there cannot be
three distinct points in $[x_0,gx_0)$ with the same image in $\quot GA$. Hence every point in
$\quot GA$ has at most two pre-images in $[x_0,gx_0)$. Pairs of points with the same
image are exactly those of the form $\{x,ax\}$ with $a\in G$ and $x\neq ax$. Call such a pair a
pair of double points.

If $[x_0,gx_0)$ has no pairs of double points then we are in case $i)$. If it has exactly a
pair of double points we are in case $ii)$. Thus, we have reduced to the case where there are
at least two pairs of double points.

Without loss of generality we may assume that $x=x_0$ and
$\{x,ax\}$ is a pair do double points. There is a second pair of double points $\{v,bv\}$
and we may assume that $v<bv$ with respect to the orientation of $[x,gx)$. Since there are
no crossing points, we have either
$$x<ax<v<bv<gx$$ or
$$x<v<bv<ax<gx$$
by interchanging the role of $x$ and $v$ we may assume we are in the first case. Since the
translation lengths of hyperbolic elements form a discrete set we may, after possibly replacing
$\{x,ax\}$ by another pair of double points, assume that either there are no pairs of double
points in $[x,ax)$ or $a$ is elliptic. Similarly for $\{v,bv\}$. By minimality of the
translation length of $g$ neither $a$ nor $b$ can have tight axis. Therefore Lemma~\ref{l:4pt}
applies and $b^1a$ has a tight axis formed by the iterated of $[x,b^{-1}ax]$.

In the order induced by the axis of $b^{-1}a$ we have
\begin{figure}[h]
\centerline{
\xymatrix{
\dots \stackrel{{x}}{\bullet}\ar@{-}[r] &\stackrel{ax}{\bullet} \ar@{-}[r] &
\stackrel{v}{\bullet}\ar@{-}[r] & \stackrel{bv}{\bullet} \ar@{-}[r] &
\stackrel{b^{-1}v}{\bullet}  \ar@{-}[r]& \stackrel{b^{-1}ax}{\bullet} \dots
}}
 \caption{The axis of $b^{-1}a$ in $A$}\label{axisba}
\end{figure}

Since there are no bad triangles in $\axis_A(g)$ the segment $[ax,v]$ projects injectively to
$\quot GA$, and the same is true for $[b^{-1}v,b^{-1}ax]$, which projects to same path in the
quotient but with opposite orientation. If there are no pairs of double points in $[x,ax)$,
then $[x,ax]$ projects to a simple closed curve. If $a$ is elliptic, it must fix the middle
point of $[x,ax]$. The similar picture holds true for $[v,bv]$. Therefore, $\axis_A(b^{-1}a)$
projects to a barbell, possibly degenerate depending on whether $a,b$ are elliptic or not.
\end{proof}

\bibliographystyle{amsplain}

\end{document}